\documentclass[12pt,a4paper]{article}
\usepackage{amsmath,amssymb,amsfonts,amsthm}
\usepackage{enumerate}
\usepackage{float}
\usepackage{longtable}
\usepackage{bussproofs}
\usepackage{hyperref}
\hypersetup{
colorlinks=true,
linkcolor=red, 
citecolor={blue},     
urlcolor=black,
pagebackref=true,
}
\makeatletter
\newcommand*{\rom}[1]{\expandafter\@slowromancap\romannumeral #1@}
\makeatother

\usepackage[left=2cm, right=2cm, top=2.5cm, bottom=2.5cm]{geometry}
\newcommand{\D}{\mathcal{D}}

\newcommand{\AXC}{\AxiomC}
\newcommand{\UIC}{\UnaryInfC}
\newcommand{\BIC}{\BinaryInfC}

\newcommand{\DP}{\DisplayProof}
\newcommand{\RL}{\RightLabel}
\newcommand{\LL}{\LeftLabel}
\newtheorem{The}{Theorem}[section]

\newtheorem{Lem}[The]{Lemma}
\newtheorem{Cor}[The]{Corollary}

\newtheorem{Rem}[The]{Remark}

\newtheorem{Examp}[The]{Example}
\newtheorem{Nota}[The]{Notation}
\let\oldproofname=\proofname
\renewcommand{\proofname}{\textit{\rm\bf\oldproofname}}
\title{\bf\Large A Cut-Free Gentzen-Style Sequent Calculus   for the Modal Logic S5
\thanks
{{\it Key Words}: Cut-free, Subformula property, Sequent calculus, Gentzen-style, Modal logic S5.}
\thanks {2010{ \it Mathematics Subject Classification}.  Primary, 03F05, 03B45.}}
\author{{\bf Mojtaba Aghaei $^{{\rm a}}$\thanks{Corresponding author.}~   and~ {\bf Hamzeh Mohammadi}$^{{\rm a}}$ } \\
{\small{ $^{{\rm a}}$Department of Mathematical Sciences,  Isfahan University of Technology}}\vspace{-1mm}\\
{\small{     Isfahan, 84156-83111,  Iran}}\\
{\small{aghaei@cc.iut.ac.ir}}\vspace{-1mm}\\
{\small{hamzeh.mohammadi@math.iut.ac.ir}}\vspace{-1mm}
}
\date\today

\begin{document}
  \maketitle

\begin{abstract}
We present the system G3{\scriptsize S5}, a Gentzen-style sequent calculus system   for the modal propositional logic S5, which in a sense has   the subformula property.  We formulate the rules of G3{\scriptsize S5} in the system $ \text{G3{\scriptsize S5}}^; $ which has the subformula property and prove the admissibility of the weakening, contraction and cut rules for it.
\end{abstract}
\section{\bf Introduction}
Sequent calculus systems for the modal logic S5  have been widely studied for a long time. Several  authors have proposed
many sequent calculus   for S5,  however, each of them presents some
difficulties (see e.g. \cite{ohnishi1957gentzen, ohnishi1959gentzen,wansing1994sequent,sato1980cut}). 
There are also  many extensions of the sequent calculus notably they are labelled sequent calculus 
(see e.g. \cite{brauner2000cut,negri2005proof,mints1997indexed,indrzejczak1998cut}),  display calculus
(see e.g. \cite{ belnap1982display, wansing1999predicate}),
hypersequent calculus 
(see e.g. \cite{ poggiolesi2008cut, restall2005proofnets, kurokawa2013hypersequent}), deep inference system (see e.g.  \cite{stouppa2007deep})
and nested sequent 
(see \cite{brunnler2009deep}). These extensions  are departing  from the Gentzen-style sequent calculus. Among these extensions, labelled and display sequent calculus are syntactically impure because of using explicit semantic parameters.

However, modal logics weaker than S5 have cut-free sequent calculus system. For example, a Gentzen-style sequent calculus for S4 is  presented   in \cite{bpt}.  It is obtained by extending G3c (The sequent calculus for classical logic in which  the structural rules of weakening and contraction are admissible) with the following rules:
\begin{align*}
&\AXC{$ \Gamma,A,\Box A\vdash \Delta $}
\RL{L$\Box $}
\UIC{$ \Gamma,\Box A\vdash\Delta $}
\DP
&&
\AXC{$ \Box\Gamma\vdash A,\Diamond\Delta $}
\RL{R$\Box $}
\UIC{$ \Gamma',\Box\Gamma\vdash\Box A,\Diamond\Delta,\Delta' $}
\DP
\\[0.15cm]
&\AXC{$ \Box\Gamma,A\vdash\Diamond\Delta $}
\RL{L$\Diamond $}
\UIC{$ \Gamma',\Box\Gamma,\Diamond A\vdash\Diamond\Delta,\Delta' $}
\DP
&&
\AXC{$ \Gamma\vdash A,\Diamond A,\Delta $}
\RL{R$\Diamond $.}
\UIC{$ \Gamma\vdash\Diamond A,\Delta $}
\DP
\end{align*}
In this paper, we provide a Gentzen-style sequent calculus system for S5,  and  call it G3{\scriptsize S5}. Similar to the sequent calculus for S4, G3{\scriptsize S5} is obtained by extending G3c with the following rules:
\begin{align*}
&\AXC{$ \Gamma,A,\Box A\vdash \Delta $}
\RL{L$\Box $}
\UIC{$ \Gamma,\Box A\vdash\Delta $}
\DP
&&
\AXC{$ M,\Diamond\bigwedge (P,\neg Q)\vdash N,A  $}
\RL{R$\Box $}
\UIC{$M,P\vdash Q,N,\Box A$}
\DP
\\[0.15cm]
&\AXC{$  A,M\vdash \Box\bigvee(\neg P,Q),N  $}
\RL{L$\Diamond $}
\UIC{$\Diamond A,M,P\vdash Q,N $}
\DP
&&
\AXC{$ \Gamma\vdash A,\Diamond A,\Delta $}
\RL{R$\Diamond $.}
\UIC{$ \Gamma\vdash\Diamond A,\Delta $}
\DP
\end{align*}
 where $M$ and $N $ are multisets of modal formulas, and $ P$ and $Q $  are multisets of atomic formulas.  In the following we prove  a simple sequent to show details of this rules.
 \begin{prooftree}
 	\AXC{$ p,\Box p\vdash p $}
 	\RL{L$ \neg $}
 	\UIC{$ \neg p,p,\Box p\vdash $}
 	\RL{L$ \Box $}
 	\UIC{$ \neg p,\Box p\vdash $}
 	\RL{L$ \Diamond $}
 	\UIC{$ \Diamond\neg p,\Box p\vdash $}
 	\RL{R$ \neg $}
 	\UIC{$ \Diamond\neg p\vdash \neg \Box p $}
 	\RL{R$ \Box $}
 	\UIC{$ \vdash p,\Box\neg \Box p $}
 	\end{prooftree}
The rule L$ \Diamond $ and R$ \Box  $ does  not have subformulas property, but the formulas in the premises are constructed from the atomic formulas in the conclusions, in this  sense the system G3{\scriptsize S5}  has  the subformula property.
Moreover, for convenience  we rewrite this system  using  semicolon (;), which 
not only has the subformula property in the strict sense but also help us to  prove  the admissibility of the  weakening, contraction and cut rules.
 
\textbf{Organization}. This paper is organized as follows. In Section \ref{sec 2}, we recall axioms of S5 and its kripke models. In Section \ref{G3S5}, we present Gentzen-style sequent calculus G3{\scriptsize S5}, and also  show that if one  uses some simpler versions  of 	 the rules 
L$\Diamond $ and R$\Box $, then  the  completeness, invertibility of the rules, and admissibility of the cut rule will not be satisfied. In addition in Subsection \ref{Another version}, for convenience,  we formulate the rules of G3{\scriptsize S5} by   semicolon  (;), in which  the system G3{\scriptsize S5} enjoys the subformula property. Using this notation,  in  Section \ref{Structural properties}, we prove the admissibility of the weakening, contraction,  the general versions of  the rules L$\Diamond $ and R$\Box $,  and some other properties of the G3{\scriptsize S5}. In Section \ref{sec cut}, besides the cut rule, a new version of cut rule is introduced, where the admissibility of each of them  concludes the admissibility of the other. We prove  the admissibility of them by induction simultaneously.
\section{\bf  Modal logic S5} \label{sec 2}
In this section, we recall  the  axiomatic formulation and the Kripke semantic of modal logic  S5.

The language of  modal logic S5 is obtained by adding to the language of
propositional logic the two modal operators $ \Box $ and $ \Diamond $. Atomic formulas are denoted by 
$ p,q,r, $ and so on.
 Formulas, denoted by 
$ A,B,C,\ldots $,
 are defined by the following grammar:
\[ A:=\bot\,|\top\,|p\,|\neg A|\,A\wedge A|\,A\vee A|\,A\rightarrow A|\,\Diamond A|\,\Box A.\]
where $ \bot $ is a  constant for falsity, and $ \top  $ is a  constant for truth.

Modal logic S5 has the following axiom schemes:
\begin{align*}
& \text{All propositional tautologies,}\\
& \text{(Dual)} \quad \Box A\leftrightarrow \neg \Diamond \neg A,\\
&\text{(K)}\quad\Box(A\rightarrow B)\rightarrow(\Box A\rightarrow \Box B),\\
&\text{(T)}\quad  \Box A\rightarrow A,\\
&\text{(5)} \quad \Diamond A\rightarrow \Box\Diamond A.
\end{align*}
Instead of (5) we can use:
\begin{align*}
& \text{(4)}\quad \Box A\rightarrow \Box\Box A,\\
& \text{(B)} \quad  A\rightarrow \Box\Diamond A.
\end{align*}
The proof rules are Modus Ponens and  Necessitation:
\begin{center}
	\AxiomC{$ A $}
	\AxiomC{$ A\rightarrow B $}
	\RightLabel{ MP,}
	\BinaryInfC{$ B $}
	\DisplayProof
	$\quad$
	\AxiomC{$A$}
	\RightLabel{N.}
	\UnaryInfC{$\Box A$}
	\DisplayProof
\end{center}
The rule Necessitation,  can be applied only to premises which are derivable in the
axiomatic system.
 If 
$ A $ is derivable in S5 from the hypotheses
$ \Gamma $, we   write $ \Gamma\vdash_{\text{S5}} A $.

A Kripke model $\mathcal{M}$ for S5 is a triple
$ \mathcal{M}=(W, R, V) $
where $ W $ is a set of states, $ R $ is an equivalence relation on $ W $, and
$ V: \varPhi\rightarrow \mathcal{P} (W) $ is a  valuation function, where
$ \varPhi $
is the set of propositional variables.
 Suppose that $ w\in W $. We inductively
define the notion of a formula $ A $ being satisfied in $\mathcal{M}$ at state $ w $ as follows:
\begin{itemize}
\item
$\mathcal{M},w \vDash p \quad \text{iff} \quad  w\in V(p),\, \text{where}\,\,
 p\in \varPhi $,
\item
$\mathcal{M},w \vDash \neg A \quad \text{iff} \quad \mathcal{M},w \nvDash  A,$
\item
$\mathcal{M},w \vDash A\vee B  \quad \text{iff} \quad \mathcal{M},w \vDash  A \,\,\text{or}\,\, \mathcal{M},w \vDash  B,$
\item
$\mathcal{M},w \vDash A\wedge B  \quad \text{iff} \quad \mathcal{M},w \vDash  A \,\,\text{and}\,\, \mathcal{M},w \vDash  B,$
\item
$\mathcal{M},w \vDash A\rightarrow B  \quad \text{iff} \quad \mathcal{M},w \nvDash  A \,\text{or}\, \mathcal{M},w \vDash  B,$
\item
$\mathcal{M},w \vDash \Diamond A \quad \text{iff} \quad  \,\,\mathcal{M},v \vDash  A\,\,\text{for some}\,\, v\in W \,\,\text{such that}\,\, R(w,v), $
\item
$\mathcal{M},w \vDash \Box A \quad \text{iff} \quad \mathcal{M},v \vDash  A \,\,\text{for all}\,\, v\in W \,\text{such that}\, R(w,v). $
\end{itemize}
Formula A is S5-valid iff it is true in every state of every S5-model.

Note that  from the point of view of the Kripke semantics, S5 is sound
and complete with respect to two different classes of frames which are equivalent. In the first
class, the accessibility relation between states of a Kripke frame is reflexive, transitive, and
symmetric (equivalently it is reflexive and Euclidean). In the second class, the accessibility
relation is absent (equivalently each two states of a Kripke frame are in relation); for more details see \cite{ Chellas, Blackburn}.
\begin{Lem}\label{2.1}
Let
$\mathcal{M}=(W,R,V)$
be a {\rm Kripke} model for
$ S5 $.
\begin{itemize}
\item[{\rm (1)}]
$\mathcal{M},w \vDash \Box A \quad \text{iff} \quad \mathcal{M},w' \vDash \Box A$, for all $ w'\in W $, where $ wRw' $.
\item[{\rm (2)}]
$\mathcal{M},w \vDash \Diamond A \quad \text{iff} \quad \mathcal{M},w' \vDash \Diamond A$, for all $ w'\in W $, where $ wRw' $.
\item[{\rm (3)}]
If
$\mathcal{M},w \vDash  A$, then  $\mathcal{M},w' \vDash \Diamond A$, for all $ w'\in W $, where $ wRw' $.
\end{itemize}
\end{Lem}
\begin{proof}
The proof clearly follows from   the definition  of satisfiability  and equivalence  of
$ R $.
\end{proof}
\section{ The Gentzen system G3{\scriptsize S5}}\label{G3S5}
The concept of a sequent is defined as usual (see, e.g. \cite{Buss,bpt, negri2008structural}).
 We  use the notation
$ \Gamma\vdash\Delta $
for sequent, where both $ \Gamma $
and
$ \Delta $ are multisets of formulas. The sequent
$ \Gamma\vdash\Delta $
is S5-valid if
$ \bigwedge\Gamma\rightarrow\bigvee\Delta $ is S5-valid.
 The notation
$ \Gamma\vdash_n\Delta $
means that 
$ \Gamma\vdash\Delta $
is derivable  with a height of derivation
at most $ n $.
\begin{Nota}\label{Nota}{\rm Throughout this paper, we  use the following notations.
\begin{itemize}
\item  
The  multisets of arbitrary formulas are denoted by  $ \Gamma,\Gamma',\Gamma_1,\Gamma_2$ and $ \Delta,\Delta', \Delta_1,\Delta_2$. The multistes of modal formulas are denoted by
$ M,M',M_1,M_2$ and $N,N',N_1,N_2 $.  The multistes of atomic formulas are denoted by $P,P', P_1,P_2,P_3$ and $Q,Q',Q_1,Q_2,Q_3 $.
\item 
The union of multisets  $ \Gamma $ and $ \Delta $  is
indicated simply by $ \Gamma,\Delta $. The union of a multiset $ \Gamma $
with a singleton multiset $ \{A\} $ is written  $ \Gamma, A $.

\item
$ \Box\Gamma=\{\Box A: A\in \Gamma\} $,
$ \Diamond\Gamma=\{\Diamond A: A\in \Gamma\} $ and
$ \neg\Gamma=\{\neg A: A\in \Gamma\} $.
\end{itemize}
}
\end{Nota}
The   system  G3{\scriptsize S5}  is given in  Table \ref{tabel 1},  which  in
the rules
R$\Box $
and
L$ \Diamond $, 
$ M$ and $N $
are multisets of modal  formulas, and
$ P$ and $Q $
are multisets of atomic formulas. The formulas $ \Diamond\bigwedge(P,\neg Q) $ in the antecedent  and $ \Box\bigvee(\neg P,Q) $ in the succedent of the premises in these rules have the same role in derivations, and equivalently can be exchanged, or be taken both of them; taking each of them, one can prove the admissibility of the others. In a bottom-up proof search, these formulas work as storage for $ P $ in the antecedent and $ Q $ in the succedent, that  in final steps  they may be  used  to get axioms (using the rules R$ \Box $ or L$ \Diamond $ with $ \Diamond\bigwedge(P,\neg Q) $ or $ \Box\bigvee(\neg P,Q) $    as principal formulas, which are probably  followed by the rule R$ \Diamond $ or L$ \Box $, see Example \ref{exa}).  The rules R$ \Box $ and L$ \Diamond $ are valid for each $ \Gamma $ and $ \Delta $ instead of $ P $ and $ Q $, and 
in Lemma \ref{General}, we show that the general versions of these rules  are  admissible  in G3{\scriptsize S5}. Since these rules do not have the subformula property in the strict sense, we restrict them in the system to multisets $ P $ and $ Q $  of atomic formulas. Because of this restriction, as mentioned above, these formulas are used as principal formulas in final steps of a bottom-up proof search.
\begin{table}[H]
	\caption{The Gentzen system G3{\scriptsize S5} for the modal logic S5\label{tabel 1}}
	\begin{center}
	\begingroup
	\def\arraystretch{2}
		\begin{tabular}{|ll|}
			\hline
			\AXC{}
			\RL{Ax}
			\UIC{$p,\Gamma\vdash \Delta,p$}
			\DP
			  &
			  \AXC{}
			  \RL{Ax}
			  \UIC{$\bot,\Gamma\vdash \Delta$}
			  \DP
			        \\
			\AXC{$ \Gamma\vdash \Delta, A $}
			\RL{L$ \neg $}
			\UIC{$ \neg A,\Gamma\vdash \Delta $}
			\DP
			&
			 \AXC{$A,\Gamma\vdash \Delta $}
			 \RL{R$\neg$}
			 \UIC{$\Gamma\vdash \Delta, \neg A $}
			 \DP
			      \\
			 \AXC{$ A,B, \Gamma\vdash \Delta $}
			 \RL{L$\wedge $}
			 \UIC{$ A\wedge B,\Gamma\vdash \Delta $}
			 \DP
			  & 
			  \AXC{$ \Gamma\vdash \Delta, A\quad \Gamma\vdash \Delta, B$}
			  \RL{R$\wedge $}
			  \UIC{$\Gamma\vdash \Delta, A \wedge B $}
			  \DP
			      \\
			 \AXC{$A, \Gamma\vdash \Delta\quad  B,\Gamma\vdash \Delta $}
			 \RL{L$\vee$}
			 \UIC{$ A\vee B,\Gamma\vdash \Delta$}
			 \DP
			  & 
			  \AXC{$\Gamma\vdash \Delta, A, B  $}
			  \RL{R$\vee $}
			  \UIC{$ \Gamma\vdash \Delta, A \vee B$}
			  \DP
			      \\
			 \AXC{$ \Gamma\vdash \Delta, A\quad  B,\Gamma\vdash \Delta $}
			 \RL{L$\rightarrow $}
			 \UIC{$A\rightarrow B,\Gamma\vdash \Delta $}
			 \DP
			  & 
			  \AXC{$ A,\Gamma\vdash \Delta, B  $}
			  \RL{R$\rightarrow$}
			  \UIC{$\Gamma\vdash \Delta, A \rightarrow B $}
			  \DP
			         \\
			 \AXC{$  A,M\vdash \Box\bigvee(\neg P,Q),N  $}
			 \RL{L$\Diamond $}
			 \UIC{$\Diamond A,M,P\vdash Q,N $}
			 \DP
			 & 
			 \AXC{$ \Gamma\vdash \Delta, \Diamond A,A  $}
			 \RL{R$\Diamond $}
			 \UIC{$\Gamma\vdash \Delta,\Diamond A $}
			 \DP
			       \\
			  \AXC{$  \Gamma, A, \Box A\vdash\Delta $}
			  \RL{L$\Box $}
			  \UIC{$ \Gamma,\Box A\vdash \Delta$}
			  \DP
			  & 
			   \AXC{$  M, \Diamond\bigwedge(P,\neg Q)\vdash N,A  $}
			  \RL{R$\Box $}
			  \UIC{$M,P\vdash Q,N,\Box A$}
			  \DP
			  \\
			\hline
		\end{tabular}
	\endgroup
	\end{center}
\end{table}
 Note that the premises in the rules L$\Diamond $ and R$\Box $ can be constructed  from the conclusions. 
In this sense, G3{\scriptsize S5} has the subformula property.
\begin{Lem}\label{arbit}
Sequents of the form
$ A\vdash A $,
with arbitrary formula 
$ A $, are derivable in {\rm G3{\scriptsize S5}}.
\end{Lem}
\begin{proof}
The proof is routine by induction on the complexity of the formula $ A $.
\end{proof}
\begin{Examp}\label{exa}{\rm
The following sequents are derivable in G3{\scriptsize S5}.
\begin{enumerate}
\item \label{exa1}
$ (\Diamond\Box(\neg p\vee q)\wedge p)\rightarrow q $
\item \label{exa2}
$ p\rightarrow(q\vee \Box\Diamond(p\wedge \neg q)) $
\item \label{exa4}
$ \Diamond (p\rightarrow \Box p) $
\item \label{exa5}
$ \Box(\Box\neg p\vee p)\rightarrow \Box(\neg p\vee \Box p) $
\end{enumerate}}
\end{Examp}
\begin{proof}
$ \, $\\
\begin{center}
1.
	    \AXC{$ \D_1 $}
	    \noLine
		\UIC{$\Box(\neg p\vee q)\vdash \Box(\neg p\vee q)$}
		\RL{L$\Diamond $}
		\UIC{$ \Diamond\Box(\neg p\vee q), p\vdash q $}
		\RL{L$\wedge $}
		\UIC{$ \Diamond\Box(\neg p\vee q)\wedge p\vdash q $}
		\RL{R$\rightarrow $}
		\UIC{$ \vdash(\Diamond\Box(\neg p\vee q)\wedge p)\rightarrow q $}
		\DP
		$ \quad $
2.
	    \AXC{$ \D_2 $}
	    \noLine
		\UIC{$\Diamond(p\wedge \neg q)\vdash \Diamond(p\wedge \neg q) $}
		\RL{R$\Box $}
		\UIC{$ p\vdash q, \Box\Diamond(p\wedge \neg q) $}
		\RL{R$\vee $}
		\UIC{$ p\vdash q\vee \Box\Diamond(p\wedge \neg q) $}
		\RL{R$\rightarrow $}
		\UIC{$ \vdash p\rightarrow (q\vee \Box\Diamond(p\wedge \neg q)) $}
		\DP
	\end{center}
where    $ \D_1 $ and $ \D_2 $ are used for a derivation by Lemma \ref{arbit}. 
\vspace{0.4cm}
\begin{center}
3.
		\AXC{}
		\RL{Ax}
		\UIC{$p,\Diamond p\vdash p, \Box p, \Diamond(p\rightarrow \Box p)$}
		\RL{R$\rightarrow$}
		\UIC{$\Diamond p\vdash p, p\rightarrow \Box p, \Diamond(p\rightarrow \Box p)$}
		\RL{R$\Diamond$}
		\UIC{$ \Diamond p\vdash p, \Diamond(p\rightarrow \Box p) $}
		\RL{R$\Box$}
		\UIC{$p\vdash \Box p, \Diamond (p\rightarrow \Box p)$}
		\RL{R$\rightarrow$}
		\UIC{$ \vdash p\rightarrow \Box p, \Diamond(p\rightarrow \Box p) $}
		\RL{R$\Diamond $}
		\UIC{$ \vdash\Diamond ( p\rightarrow \Box p) $}
		\DP
		\end{center}
	\vspace{0.4cm}
\begin{prooftree}
		\AXC{$  $}
		\LL{4.}
		\RL{Ax}
		\UIC{$ p,\Box\neg p,\Box(\Box\neg p\vee p)\vdash  \Box p,p$}
		\RL{L$\neg ,R \neg $}
	\UIC{$\neg p,\Box\neg p,\Box(\Box\neg p\vee p)\vdash \neg p,\Box p$}
	\RL{L$\Box$}
	\UIC{$\Box\neg p,\Box(\Box\neg p\vee p)\vdash \neg p,\Box p$}
	   \AXC{$ \D_4 $}
	   \noLine  
	   	\UIC{$\Box(\Box\neg p\vee p),\Diamond(p\wedge p)\vdash p$}
	   \RL{R$\Box$} 
\UIC{$p,p,\Box(\Box\neg p\vee p)\vdash \Box p$}
\RL{R$\neg $}
\UIC{$p,\Box(\Box\neg p\vee p)\vdash \neg p,\Box p$}
\RL{L$\vee$}
\BIC{$\Box\neg p\vee p,\Box(\Box\neg p\vee p)\vdash\neg p,\Box p$}
\RL{L$\Box$}
\UIC{$\Box(\Box\neg p\vee p)\vdash \neg p, \Box p$}
\RL{R$\vee $}
\UIC{$\Box(\Box\neg p\vee p)\vdash \neg p\vee \Box p$}
\RL{R$\Box $}
\UIC{$\Box(\Box\neg p\vee p)\vdash \Box(\neg p\vee \Box p)$}
\RL{R$\rightarrow$,}
\UIC{$\vdash\Box(\Box\neg p\vee p)\rightarrow \Box(\neg p\vee \Box p)$}
	\end{prooftree}
where $ \D_4 $ is as follows
\begin{prooftree}
	 \AXC{}
	\RL{Ax}
	\UIC{$\Box\neg p,\Box(\Box\neg p\vee p),p, p\vdash \Box p,p$}
	\RL{L$\Box $,L$ \neg $}
	\UIC{$\Box\neg p,\Box(\Box\neg p\vee p),p, p\vdash \Box p$}
	\RL{L$\wedge $}
	\UIC{$\Box\neg p,\Box(\Box\neg p\vee p),p\wedge p\vdash \Box p$}
	\RL{L$\Diamond $} 
	\UIC{$\Box\neg p,\Box(\Box\neg p\vee p),\Diamond(p\wedge p)\vdash p$}
	\AXC{}
	\RL{Ax}
	\UIC{$p,\Box(\Box\neg p\vee p),\Diamond(p\wedge p)\vdash p$}
	\RL{L$\vee$}
	\BIC{$\Box\neg p\vee p,\Box(\Box\neg p\vee p),\Diamond(p\wedge p)\vdash p$}
	\RL{L$\Box$.}
 	\UIC{$\Box(\Box\neg p\vee p),\Diamond(p\wedge p)\vdash p$}
	\end{prooftree}
\end{proof}
In the following remark, some cases of the rules L$\Diamond $ and R$\Box $ are expressed that if they are  used instead of the rules L$\Diamond $ and R$\Box $ in G3{\scriptsize S5}, then the system is   not  complete.
\begin{Rem}{\rm
Consider the following cases of the rules L$\Diamond $ and R$\Box $.
\begin{enumerate} 
\item[(1)]\label{one}
\AXC{$ A,M,\Diamond P\vdash \Box Q, N $}
\RL{L$\Diamond $}
\UIC{$ \Diamond A,M, P\vdash  Q, N $}
\DP
$ \quad $
\AXC{$ M,\Diamond P\vdash \Box Q, N,A $}
\RL{R$\Box $}
\UIC{$ M, P\vdash  Q, N,\Box A $}
\DP
\item[(2)]\label{two}
\AXC{$ A,M,\Diamond \Gamma\vdash \Box \Delta, N $}
\RL{L$\Diamond $}
\UIC{$ \Diamond A,M, \Gamma\vdash \Delta, N $}
\DP
$ \quad $
\AXC{$ M,\Diamond \Gamma\vdash \Box\Delta, N,A $}
\RL{R$\Box $}
		 \UIC{$M, \Gamma\vdash  \Delta, N,\Box A $}
		 \DP
		 \item[(3)]\label{three}
		 	\AXC{$ A,M,\Diamond \bigwedge P\vdash \Box \bigvee Q, N $}
		 \RL{L$\Diamond $}
		 \UIC{$ \Diamond A,M, P\vdash  Q, N $}
		 \DP
		 $ \quad $
		 \AXC{$ M,\Diamond \bigwedge P\vdash \Box \bigvee Q, N,A $}
		 \RL{R$\Box $}
		 \UIC{$ M, P\vdash  Q, N,\Box A $}
		 \DP
		 \item[(4)]\label{four}
		 \AXC{$ A,M,\Diamond \bigwedge \Gamma\vdash \Box \bigvee \Delta, N $}
		 \RL{L$\Diamond $}
		 \UIC{$ \Diamond A,M, \Gamma\vdash  \Delta, N $}
		 \DP
		 $ \quad $
		 \AXC{$ M,\Diamond \bigwedge \Gamma\vdash \Box \bigvee \Delta, N,A $}
		 \RL{R$\Box $}
		 \UIC{$ M, \Gamma\vdash  \Delta, N,\Box A $}
		 \DP
		\end{enumerate}
If we use each of the above cases instead of 	 the rules 
L$\Diamond $ and R$\Box $ in G3{\scriptsize S5}, then  we do not have  completeness, invertibility of the rules, and admissibility of the cut rule.
For example,  the items 1 and 2  of Example \ref{exa}  are not provable in all cases. For more details,
consider the following examples:
		 	\begin{itemize}
		 	\item[(a)]$\vdash p\rightarrow (q\rightarrow(\Box\Diamond(p\wedge q))$
		 	\item[(b)]$\vdash(p\wedge q)\rightarrow \Box\Diamond(p\wedge q)$
		 	\item[(c)]$\vdash\Diamond\Box(p\vee q)\rightarrow (p\vee q)$
		 \end{itemize}
In the case (1), Examples (a), (b) and (c) are  not provable, and the rule
		 \begin{prooftree}
		 	\AXC{$ p\wedge q\vdash \Box\Diamond(p\wedge q) $}
		 	\RL{L$\Diamond $}
		 	\UIC{$ \Diamond(p\wedge q)\vdash \Box\Diamond(p\wedge q) $}
		 \end{prooftree}
		 is not invertible, thus we do not have the cut rule
		 \begin{prooftree}
		 	\AXC{$ p\wedge q\vdash p\wedge q $}
		 	\RL{R$\Diamond $}
		 	\UIC{$ p\wedge q\vdash \Diamond (p\wedge q) $}
		 	\AXC{$ \Diamond(p\wedge q)\vdash \Diamond (p\wedge q) $}
		 	\RL{R$\Box $}
		 	\UIC{$ \Diamond(p\wedge q)\vdash \Box\Diamond (p\wedge q) $}
		 	\RL{Cut.}
		 	\BIC{$ p\wedge q\vdash \Box\Diamond(p\wedge q) $}
		 \end{prooftree}
In the case (2), Example (a)  is not provable  but (b) and (c) are provable, and  
the rules L$\wedge $ and R$\vee $:
		 \begin{center}
		 	\AXC{$ p, q\vdash \Box\Diamond(p\wedge q) $}
		 	\UIC{$ p\wedge q	\vdash \Box\Diamond(p\wedge q) $}
		 	\DP
		 	$ \quad $
		 	\AXC{$ \Diamond\Box(p\vee q)\vdash p, q $}
		 	\UIC{$ \Diamond\Box(p\vee q)\vdash p\vee q $}
		 	\DP
		 	\end{center}
are not invertible  thus we do not have the cut rules
		 \begin{prooftree}
		 \AXC{$ p,q\vdash p\wedge q $}
		 \AXC{$ p\wedge q\vdash \Box\Diamond(p\wedge q) $}
		 \BIC{$ p,q\vdash \Box\Diamond (p\wedge q) $}
		 \end{prooftree}
and
		 \begin{prooftree}
		 	\AXC{$ \Diamond\Box (p\vee q)\vdash p\vee q $}
		 	\AXC{$ p\vee q\vdash p,q $}
		 	\RL{.}
		 	\BIC{$ \Diamond\Box(p\vee q)\vdash p,q $}
		 \end{prooftree}
In the cases (3) and (4), Examples (a), (b) and (c) are provable, and since
		  $ \Diamond\Box(\neg p\vee q),p\vdash q $
		  is not provable in them but we have
		\begin{prooftree}
			\AXC{$ \Box(\neg p\vee q)\vdash \Box (\neg p\vee q) $}
			\RL{L$\Diamond $}
			\UIC{$ \Diamond\Box(\neg p\vee q)\vdash \neg p,q $}
			\end{prooftree}
and
		\begin{prooftree}
			\AXC{$ \Box(\neg p\vee q)\vdash \Box(\neg p\vee q) $}
			\RL{L$ \Diamond$}
			\UIC{$ \Diamond\Box(\neg p\vee q)\vdash \neg p,q $}
			\AXC{$ p\vdash p $}
			\RL{L$\neg $}
			\UIC{$ \neg p ,p\vdash $}
			\RL{Cut}
			\BIC{$ \Diamond\Box(\neg p\vee q),p\vdash q $}
			\end{prooftree}
thus R$\neg $ is not invertible and  the cut rule is not admissible.		
	}
\end{Rem}
\subsection{ The system $ \text{\rm G3{\scriptsize S5}}^;  $ \label{Another version}}
In this subsection,  we present $ \text{\rm G3{\scriptsize S5}}^;  $ another version of G3{\scriptsize S5}  which is a rewriting of sequents by using  semicolon (;). This system  not only has the subformula property in the strict sense but also helps us to  prove  the admissibility of the  weakening, contraction and cut rules. 

Let $ P $ and $ Q $ be multisets of atomic formulas,  $ \Gamma $ and $ \Delta $ be multisets of arbitrary formulas, and let
$ {\cal P}=\{X_1,\ldots, X_n, Y_1,\ldots, Y_m\} $
 be a partition of $ P\bigcup Q $, and $ X'_i=\{p\in P: p\in X_i\}\bigcup \{\neg q: q\in Q,  q\in X_i\} $,  $\, Y'_j=\{\neg p: p\in P,  p\in Y_j\}\bigcup \{q\in Q:  q\in Y_j\} $, for $ i=1,\ldots,n $ and $ j=1,\ldots,m $.
 We denote the  sequent 
$
 \Gamma, \Diamond\bigwedge X'_1,\ldots,\Diamond\bigwedge X'_n\vdash \Box\bigvee Y'_1,\ldots,\Box\bigvee Y'_m,\Delta
$
 by $ \Gamma;P\vdash^{\cal{P}} Q;\Delta $, or  by $ \Gamma;P\vdash Q;\Delta $.

 We say that two occurrences  of  formulas in the middle part (between two semicolons) are related if they are in the same $ X_i $ or $ Y_j $ in the partition. A relatedness of the middle part of the sequent $ \Gamma;P\vdash Q;\Delta $ is an equivalent relation corresponding to a partition of $ P\bigcup Q $ as above.  The corresponding formula of the formulas in  $ X_i $ is $ \Diamond\bigwedge X'_i $, and the corresponding formula of the formulas in  $ Y_j $ is $ \Box\bigvee Y'_j $.

Note that a multiset $ X $ is called a submultiset of $ \Gamma $ if $ M_X(A)\leq M_\Gamma(A) $ for all $ A\in X $, where $ M_X(A) $ and $ M_\Gamma(A) $ are the multiplicities of $ A $ in multisets $ X $ and $ \Gamma $, respectively.

 If  $ P $ or $ Q $ is empty,  we will avoid to write  its corresponding semicolon.

Each sequent
$ \Gamma_1\vdash\Delta_1 $
can be   written variously in the form $ \Gamma;P\vdash Q;\Delta $. For example,  the  sequent
$\Gamma,\Diamond q, \Diamond(q\wedge \neg r)\vdash\Box(\neg q\vee r), \Box(\neg p\vee r),\Box r,\Delta$
 is denoted by  each of the following
 \begin{itemize}
 \item[(1)]   $ \Gamma,\Diamond q; q \vdash r;\Box(\neg q\vee r), \Box(\neg p\vee r),\Box r,\Delta  $
 \item[(2)] $ \Gamma,\Diamond q; q,q \vdash r,r; \Box(\neg p\vee r),\Box r,\Delta  $
\item[(3)]  $ \Gamma,\Diamond q; p,q,q \vdash r,r,r; \Box r,\Delta  $
\item[(4)] $ \Gamma; q,p,q,q \vdash r,r,r,r; \Delta $
 \end{itemize}

In the sequent (1), $ {\cal P}=\{X_1\},  $ $ X_1=\{q, r\} $, $ X'_1=\{q,\neg r\} $,  $ q $ and $ r $ in the middle part  are related,  and their corresponding  formula, $ \Diamond\bigwedge X'_1 $, is $ \Diamond(q\wedge \neg r) $.

 In the sequent (2), $ {\cal P}=\{X_1,Y_1\}$, $ Y_1=\{ q, r\} $, $ Y'_1=\{ \neg q, r\} $,  the new occurrences of $ q $ and $ r $ in the middle part  are related  and their corresponding  formula, $ \Box\bigvee Y'_1 $, is $ \Box(\neg q\vee r) $. Similarly in (3), (4) and (5),  for  $ {\cal P}=\{X_1, X_2,Y_1\}$,  $ {\cal P}=\{X_1, X_2,Y_1,Y_2\}$ and  $ {\cal P}=\{X_1, X_2,Y_1,Y_2, Y_3\}$, where $ X_2=\{q\} $, $ Y_2=\{\neg p,r\} $ and $Y_3=\{r\} $, we have the corresponding formulas $ \Diamond q $, $ \Box(\neg p\vee r) $, and $ \Box r $, respectively.
The premises and conclusion of every rule can be  rewritten  by this notation, and
 the correspondence between the original and rewritten sequents and the relatedness between formulas in the middle parts  in the premises  can be determined from the   correspondence  between them  and the relatedness in the conclusion. Thus by  determining the correspondence and relatedness in the root of a derivation (usually we take  the middle part as the empty set), all correspondences and relatednesses  during a derivation are determined, and we will not encounter any  ambiguity.

Therefore,  we rewrite the rules of  G3{\scriptsize S5} in  Table \ref{tabel 2}, where the premise and conclusion of any rule except L$\Diamond $ and R$\Box $  have the same formulas in the middle parts  and the same corresponding formulas and relatedness.  By applying the rules  L$\Diamond $ and R$\Box $, the atomic formulas in the conclusion move to the middle part in the premise. For example, the corresponding formulas in the rule L$ \Diamond $ are as
\begin{prooftree}
	\AXC{$ M, A, \Diamond\bigwedge X_1,\ldots, \Diamond\bigwedge X_n \vdash \Box\bigvee Y_1,\ldots,\Box\bigvee Y_m, \Box\bigvee(\neg P_1, Q_1), N  $}
	\RL{L$\Diamond $,}
	\UIC{$ M,P_1, \Diamond A, \Diamond\bigwedge X_1,\ldots, \Diamond\bigwedge X_n\vdash \Box\bigvee Y_1,\ldots,\Box\bigvee Y_m,  Q_1, N  $}
\end{prooftree}
where 
$ X_1,\ldots, X_n  $ and
$ Y_1,\ldots, Y_m  $
are submultisets of 
$ P_2,\neg Q_2 $
and
$ \neg P_2,  Q_2 $, respectively, and  $ X'_1,\ldots, X'_n, Y'_1,\ldots, Y'_m $ is the partition of $ P_2,Q_2 $, the middle part of the conclusion. The partition of the middle part of the premise is  $ X'_1,\ldots, X'_n, Y'_1,\ldots, Y'_m, Y'_{m+1} $, where $ Y_{m+1}=\neg P_1, Q_1 $. The formulas in $ P_1 $, $ Q_1 $ are related in the premise with corresponding formula $ \Box\bigvee(\neg P_1, Q_1) $, and the relatedness between formulas in $ P_2,Q_2 $  and their corresponding formulas in the premise is the same as in the conclusion. \\The rule R$ \Box $ is similar, except we have $ X_{n+1}= P_1, \neg Q_1 $  with corresponding formula $ \Diamond\bigwedge(P_1, \neg Q_1) $ in the premise.
\begin{table}[H]
	\caption{The Gentzen system $ \text{\rm G3{\scriptsize S5}}^;  $  \label{tabel 2}}
		\begin{small}	
			\begin{center}
					\begingroup
				\def\arraystretch{2}
		\begin{tabular}{|ll|}
			\hline
			\AXC{}
			\RL{Ax}
			\UIC{$p,\Gamma;P\vdash Q; \Delta,p $}
			\DP
			&
			\AXC{}
			\RL{Ax}
			\UIC{$\bot,\Gamma;P\vdash Q; \Delta $}
			\DP
             \\
			\AXC{$ \Gamma;P\vdash Q; \Delta,A $}
			\RL{L$\neg $}
			\UIC{$ \neg A,\Gamma;P\vdash Q; \Delta $}
			\DP
			&
			\AXC{$ A,\Gamma;P\vdash Q; \Delta $}
			\RL{R$\neg $}
			\UIC{$ \Gamma;P\vdash Q; \Delta,\neg A$}
			\DP
              \\
			\AXC{$ A,\Gamma;P\vdash Q; \Delta $}
			\AXC{$  B,\Gamma;P\vdash Q; \Delta $}
			\RL{L$\vee $}
			\BIC{$ A\vee B,\Gamma;P\vdash Q; \Delta $}
             \DP
             &
             	\AXC{$ \Gamma;P\vdash Q; \Delta, A,B $}
             \RL{R$\vee $}
             \UIC{$\Gamma;P\vdash Q; \Delta,A\vee B $}
             \DP
             \\
             	\AXC{$ A,B,\Gamma;P\vdash Q; \Delta $}
             \RL{L$\wedge $}
             \UIC{$A\wedge B,\Gamma;P\vdash Q; \Delta $}
             \DP
             &
			\AXC{$ \Gamma;P\vdash Q; \Delta,A $}
			\AXC{$ \Gamma;P\vdash Q; \Delta,B $}
			\RL{R$\wedge $}
			\BIC{$ \Gamma;P\vdash Q; \Delta,A\wedge B $}
              \DP
              \\
			\AXC{$ \Gamma;P\vdash Q; \Delta, A $}
			\AXC{$ B,\Gamma;P\vdash Q; \Delta $}
			\RL{L$\rightarrow $}
			\BIC{$ A\rightarrow B,\Gamma;P\vdash Q; \Delta $}
            \DP
              &
			\AXC{$ A,\Gamma;P\vdash Q; \Delta,B $}
			\RL{R$\rightarrow $}
			\UIC{$ \Gamma;P\vdash Q; \Delta,A\rightarrow B $}
           \DP
            \\
			\AXC{$  A,M;P_1,P_2\vdash Q_2,Q_1;N $}
			\RL{L$\Diamond $}
			\UIC{$\Diamond A,M,P_1;P_2\vdash Q_2;Q_1,N $}
			\DP
			&
			\AXC{$\Gamma;P\vdash Q; \Delta,\Diamond A,A $}
			\RL{R$\Diamond $}
			\UIC{$ \Gamma;P\vdash Q; \Delta,\Diamond A$}
			\DP
             \\
			\AXC{$ A,\Box A,\Gamma;P\vdash Q;\Delta $}
			\RL{L$\Box $}
			\UIC{$\Box A,\Gamma;P\vdash Q;\Delta $}
			\DP
			&
			\AXC{$ M;P_1,P_2\vdash Q_2,Q_1;N,A$}
			\RL{R$\Box $}
			\UIC{$M,P_1;P_2\vdash Q_2;Q_1,N,\Box A $}
			\DP
            \\
			\AXC{$M,P_2;P_1,P_3\vdash Q_3, Q_1;Q_2,N  $}
			\RL{ L$\Diamond^; $ }
			\UIC{$M,P_1;P_2,P_3\vdash Q_3, Q_2;Q_1,N  $}
            \DP
             &
             \AXC{$M,P_2;P_1,P_3\vdash Q_3, Q_1;Q_2,N  $}
             \RL{ R$\Box^; $ }
             \UIC{$M,P_1;P_2,P_3\vdash Q_3, Q_2;Q_1,N  $}
             \DP
             \\
			\hline
		\end{tabular}
	\endgroup
		\end{center}
	\end{small}
\end{table}
The two last rules,
L$\Diamond^; $
and
R$\Box^; $,
seem to be the same, but they are versions of the rules L$ \Diamond $ and R$ \Box $ and in
L$\Diamond^; $
the principal formula is
$ \Diamond\bigwedge(P_2,\neg Q_2) $,
and
in
R$\Box^; $
the principal formula is
$ \Box\bigvee(\neg P_2, Q_2) $, which are the corresponding formulas for $ P_2 $ and $ Q_2 $ in the middle parts, as
\begin{prooftree}
	\AXC{$ M, P_2, \Diamond\bigwedge X'_1,\ldots, \Diamond\bigwedge X'_n\vdash \Box\bigvee Y'_1,\ldots,\Box\bigvee Y'_m,\Box \bigvee(\neg P_1,  Q_1),Q_2,  N  $}
	\UIC{$ M, \bigwedge(P_2, \neg Q_2), \Diamond\bigwedge X'_1,\ldots, \Diamond\bigwedge X'_n \vdash \Box\bigvee Y'_1,\ldots,\Box\bigvee Y'_m,\Box\bigvee(\neg P_1, Q_1),N  $}
	\RL{L$\Diamond^; $}
	\UIC{$ M,P_1, \Diamond \bigwedge(P_2, \neg Q_2), \Diamond\bigwedge X'_1,\ldots, \Diamond\bigwedge X'_n\vdash \Box\bigvee Y'_1,\ldots,\Box\bigvee Y'_m, Q_1, N  $}
	\end{prooftree}
and 
\begin{prooftree}
	\AXC{$ M, P_2, \Diamond\bigwedge( P_1,\neg Q_1),  \Diamond\bigwedge X'_1,\ldots, \Diamond\bigwedge X'_n\vdash \Box\bigvee Y''_1,\ldots,\Box\bigvee Y''_m,Q_2,  N  $}
	\UIC{$ M, \Diamond\bigwedge( P_1,\neg Q_1), \Diamond\bigwedge X'_1,\ldots, \Diamond\bigwedge X'_n \vdash \Box\bigvee Y''_1,\ldots,\Box\bigvee Y''_m,\bigvee(\neg P_2,  Q_2),N  $}
	\RL{R$\Box^; $}
	\UIC{$ M,P_1,  \Diamond\bigwedge X'_1,\ldots, \Diamond\bigwedge X'_n\vdash \Box\bigvee Y''_1,\ldots,\Box\bigvee Y''_m, Q_1,\Box\bigvee(\neg P_2, Q_2), N  $}
\end{prooftree}
where 
$ X_1,\ldots, X_n,   Y_1,\ldots, Y_m  $
is the partition  of
$ P_3\bigcup Q_3 $.

In the rule L$ \Diamond^; $, the partition of $ P_2, P_3,Q_3,Q_2 $,  the middle part of the conclusion, is 
\[ X_1,\ldots, X_n,X_{n+1}, Y_1,\ldots, Y_m,\]
 where $ X_{n+1}= P_2\bigcup  Q_2 $. The partition of the middle part of the premise is  $$ X_1,\ldots, X_n, Y_1,\ldots, Y_m, Y_{m+1}, $$ where $ Y_{m+1}=P_1\bigcup Q_1 $. The formulas in $ P_1 $, $ Q_1 $ are related in the premise with corresponding formula $ \Box\bigvee(\neg P_1, Q_1) $, and the relatedness between formulas in $ P_3,Q_3 $  and their corresponding formulas in the premise is the same as in the conclusion. 
 
 In the rule R$ \Box^; $, the partition of $ P_2, P_3,Q_3,Q_2 $,  the middle part of the conclusion, is 
  $$ X'_1,\ldots, X'_n, Y'_1,\ldots, Y'_m, Y'_{m+1}, $$
 where $ Y_{m+1}= \neg P_2,  Q_2 $. The partition of the middle part of the premise is 
 \[ X'_1,\ldots, X'_n,X'_{n+1}, Y'_1,\ldots, Y'_m,\]
   where $ X_{n+1}= P_1, \neg Q_1 $. The formulas in $ P_1 $, $ Q_1 $ are related in the premise with corresponding formula $ \Diamond\bigwedge( P_1,\neg Q_1) $, and the relatedness between formulas in $ P_3,Q_3 $  and their corresponding formulas in the premise is the same as in the conclusion.
\begin{Rem}
	We can consider the system $ \text{\rm G3{\scriptsize S5}}^;  $  primary. As mentioned above the partition  and relatedness of formulas in the middle part of the premises can be determined from theirs in the conclusion of the rules. Then there is no ambiguity in determining  the partition of sequents from the partition in the root, that is usually empty. Therefore it is possible  to remove the partitions  $ {\cal P} $ during a derivation.
	\end{Rem}
 For example, we  prove $\vdash (r\wedge p)\rightarrow (q\rightarrow \Box(\Diamond (p\wedge q)\wedge \Diamond r)) $ in this system as follows.  
\begin{prooftree}
	\AXC{}
	\RL{Ax}
	\UIC{$ p,q,r\vdash \Diamond(p\wedge q),p $}
	\AXC{}
	\RL{Ax}
	\UIC{$ p,q,r\vdash \Diamond(p\wedge q),q $}
	\RL{R$\wedge $}
	\BIC{$ p,q,r\vdash \Diamond(p\wedge q), p\wedge q $}
	\RL{R$\Diamond $}
	\UIC{$ p,q,r\vdash \Diamond(p\wedge q) $}
	\RL{L$\Diamond^; $}
	\UIC{$ ;r,p,q\vdash \Diamond(p\wedge q) $}
	    \AXC{}
	    \RL{Ax}
	    \UIC{$ r,p,q\vdash \Diamond r,r $}
	    \RL{R$\Diamond $}
	    \UIC{$ r,p,q\vdash \Diamond r $}
	    \RL{L$\Diamond^; $}
	    \UIC{$ ;r,p,q\vdash \Diamond r $}
   \RL{R$\wedge $}
   \BIC{$ ;r,p,q\vdash\Diamond (p\wedge q)\wedge \Diamond r $}
   \RL{R$\Box $}
   \UIC{$ r,p,q\vdash\Box(\Diamond (p\wedge q)\wedge \Diamond r) $}
   \RL{L$\wedge $}
   \UIC{$ r\wedge p,q\vdash  \Box(\Diamond (p\wedge q)\wedge \Diamond r) $}
   \RL{R$\rightarrow $}
   \UIC{$ r\wedge p\vdash q\rightarrow \Box(\Diamond (p\wedge q)\wedge \Diamond r) $}
   \RL{R$\rightarrow $}
   \UIC{$\vdash (r\wedge p)\rightarrow (q\rightarrow \Box(\Diamond (p\wedge q)\wedge \Diamond r)) $}
\end{prooftree}
The Case  \ref{exa5} of Example \ref{exa} is proved as follows:
\begin{prooftree}
	\ref{exa5}.
   \AXC{}
   \RL{Ax}
   \UIC{$\Box\neg p,\Box(\Box\neg p\vee p),p\vdash p;p $}
   \RL{L$\neg $}
   \UIC{$\neg p,\Box\neg p,\Box(\Box\neg p\vee p),p\vdash p; $}
   \RL{L$\Box $}
   \UIC{$ \Box\neg p,\Box(\Box\neg p\vee p),p\vdash  p; $}
   \RL{L$\Diamond^; $}
   \UIC{$ \Box\neg p,\Box(\Box\neg p\vee p);p\vdash p $}
   \AXC{}
   \RL{Ax}
   \UIC{$ p,\Box(\Box\neg p\vee p);p\vdash p $}
   \RL{L$\vee $}
   \BIC{$ \Box\neg p\vee p,\Box(\Box\neg p\vee p);p\vdash p $}
   \RL{L$\Box, $}
   \UIC{$ \Box(\Box\neg p\vee p);p\vdash p $}
   \RL{R$\Box $}
   \UIC{$ \Box(\Box\neg p\vee p),p\vdash  \Box p $}
   \RL{R$\neg $}
   \UIC{$ \Box(\Box\neg p\vee p)\vdash \neg p,\Box p $}
   \RL{R$\vee $}
   \UIC{$ \Box(\Box\neg p\vee p)\vdash \neg p\vee\Box p $}
   \RL{R$\Box $}
   \UIC{$ \Box(\Box\neg p\vee p)\vdash \Box(\neg p\vee\Box p) $}
\end{prooftree}
Note that in the above  derivations, going upward, $ r,p,q $ together  in the first and  $ p $ in the latter move into the middle parts in the rules R$ \Box $ and move out in the rules L$ \Diamond^; $.

Using this notation, we prove the  admissibility of the cut rule (Theorem
\ref{cut}) and its special cases  the weakening rule (Lemma \ref{W}) and the contraction rule (Lemma \ref{C}).

\begin{The}[Soundness]\label{Soundness}
	If
	$ \Gamma\vdash\Delta $
	is provable in $ \text{\rm G3{\scriptsize S5}}^;  $, then it is S5-valid.
\end{The}
\begin{proof}
	The proof is by induction on the height of the derivation of $ \Gamma\vdash\Delta $ in $ \text{\rm G3{\scriptsize S5}}^;  $. Initial sequents are obviously valid
	in  every Kripke model for S5.  We only check the induction step for the rule  R$ \Box $, the other cases can be verified similarly. \\
	Suppose that the sequent $ \Gamma\vdash\Delta $ is $ M,P_1;P_2\vdash Q_2;Q_1,N,\Box A $, the conclusion of rule R$ \Box $, with the premise $ M;P_1,P_2\vdash Q_2,Q_1;N,A $, and assume by the induction hypothesis that the premise is valid in every Kripke model for S5.
	By contradiction assume  that the conclusion is not S5-valid. We assume without loss of generality that $ M,P_1;P_2\vdash Q_2;Q_1,N,\Box A $ stands for 
	$ M,P_1,\Diamond\bigwedge(P_2,\neg Q_2)\vdash Q_1, N,\Box A$. Therefore, there is a Kripke model $ \mathcal{M}=(W,R,V) $  such that 
	\begin{align}
	&\mathcal{M},w\vDash \bigwedge\left( M, P_1,\Diamond\bigwedge(P_2,\neg Q_2)\right)   \label{1} \\
	& \mathcal{M},w\nvDash \bigvee\left(  Q_1, N,\Box A\right),\label{2}
	\end{align}
	for some $ w\in W $.
	Also, since the premise is valid,	in all kripke models  like $ \mathcal{M} $ we have
	\begin{equation}\label{3}
	\text{If}\,\, \mathcal{M},w \vDash \bigwedge\left( M, \Diamond\bigwedge(P_2,\neg Q_2)\right)  \,\, \text{then}\,\,   \mathcal{M},w \vDash \bigvee\left( \Box\bigvee(\neg P_1,Q_1), N, A\right). 
	\end{equation}
	By (\ref{1}) and (\ref{3}), we have the following cases:
	\begin{itemize}
		\item[(a)]
		$ \mathcal{M},w\vDash\Box\bigvee(\neg P_1,Q_1) $
		\item[(b)]
		$ \mathcal{M},w\vDash \bigvee N $
		\item[(c)]
		$ \mathcal{M},w\vDash A $
	\end{itemize}
	The case (a) is a  contradiction with (\ref{1}) and (\ref{2}). The cases (b)  will contradicts with (\ref{2}). Finally, if   $ \mathcal{M},w\vDash A $, then we  will show that it is also a contradiction with (\ref{2}). Since  $ M $ is the multiset of modal formulas  by Lemma \ref{2.1},
	$ \mathcal{M},w' \vDash \bigwedge\left( M, \Diamond\bigwedge(P_2,\neg Q_2)\right)  $ for all $ w'\in W $, where $ wRw' $. Then  by (\ref{3}), we have
	$$  \mathcal{M},w' \vDash \bigvee\left( \Box\bigvee(\neg P_1,Q_1), N, A\right) $$
	If $  \mathcal{M},w' \vDash \bigvee\left( \Box\bigvee(\neg P_1,Q_1), N\right) $, then since $ N $  is the multiset of modal formulas again  by  Lemma \ref{2.1} 
	$$  \mathcal{M},w \vDash \bigvee\left( \Box\bigvee(\neg P_1,Q_1), N\right),$$
	which contradicts with (\ref{1}) and (\ref{2}).  So $ \mathcal{M},w'\vDash  A $ for all $ w'\in W $, where $ wRw' $,  hence 
	$ \mathcal{M},w\vDash \Box A $ which also contradicts with (\ref{2}).
\end{proof}

\section{Structural properties}\label{Structural properties}
In this section, we show that weakening and contraction rules   are admissible in {\rm G3{\scriptsize S5}}.  We remove the proofs of some lemmas since they are easy or similar
to the proofs in \cite{negri2008structural, bpt}.

A rule of $ \text{\rm G3{\scriptsize S5}}^;  $ is said to be height-preserving invertible if whenever an instance of its conclusion is
derivable in $ \text{\rm G3{\scriptsize S5}}^;  $ with height $ n $, then so is the corresponding instance of its premise(s).
\begin{Lem}[Inversion Lemma]	\label{inversion}
All rules of $ \text{\rm G3{\scriptsize S5}}^;  $,  with the exception of 
L$\Diamond $ and
R$\Box $,
are height-preserving invertible.
\end{Lem}
\begin{proof}
The proof is by induction on the height of derivations.
\end{proof}
\begin{Lem} \label{W}
The rule of weakening,
	\begin{prooftree}
		\AXC{$ \Gamma;P\vdash Q;\Delta $}
		\RL{W,}
		\UIC{$ \Gamma',\Gamma;P',P\vdash Q,Q';\Delta,\Delta' $}
	\end{prooftree}
is admissible,
where 
$ \Gamma' $ and 
$ \Delta' $ are
multisets of arbitrary formulas, and 	$ P' $ and 
$ Q' $ are 
multisets of atomic formulas.
\end{Lem}
\begin{proof}
{\rm
The proof is by induction on the height of derivation of  the premise.   Let $ \D $ be a derivation of  
$ \Gamma;P\vdash Q;\Delta$. We consider only the case  when the last rule in
$ \D $
is 
R$\Box $, since   
L$\Diamond $  is treated symmetrically. For the remaining rules  it is sufficient to apply the induction hypothesis to the premise(s) and
then use the same rule to obtain conclusion.
		
Let the last rule  be R$\Box $. The proof is by subinduction on the  complexity of formulas in $ \Gamma' $ and $ \Delta' $. We just prove the case when 
$ \Gamma' $ and $ \Delta' $ contains only atomic and modal formulas, since  derivations of other cases are constructed by this case. So suppose that 
$ \Gamma'= M'_1,P'_1$ and $ \Delta'=N'_1,Q'_1 $,
where $ M'_1 $ and $ N'_1 $ are multisets of modal formulas, and $ P'_1 $ and $ Q'_1$ are multisets of atomic formulas.
There are two subcases, according to the position of the principal formula.
		 	
Subcase 1.	Let
$ \Gamma=M_1,P_1 $ and $ \Delta=Q_1,N_1, \Box A $,  and let
$\Box A $ be the principal formula. Assume $ \D $ is as 
		\begin{prooftree}
			\AXC{$ \D_1 $}
			\noLine
			\UIC{$M_1;P_1,P\vdash Q,Q_1;N_1, A$}
			\RL{R$\Box $.}
			\UIC{$M_1,P_1;P\vdash Q;Q_1,N_1,\Box A $}
		\end{prooftree}
Then we have
		\begin{prooftree}
			\AXC{$ \D_1 $}
			\noLine
			\UIC{$M_1;P_1,P\vdash Q,Q_1;N_1, A$}
			\RL{IH}
			\UIC{$M'_1,M_1;P'_1,P_1,P',P\vdash Q,Q',Q_1,Q'_1;N_1, A,N'_1$}
			\RL{R$\Box $.}
			\UIC{$M'_1,P'_1,M_1,P_1;P',P\vdash Q,Q';Q_1,N_1, \Box A,Q'_1,N'_1$}
		\end{prooftree}
	
Subcase 2.		Let
	$ \Gamma=M_1,P_1 $ and $ \Delta=Q_1,N_1 $,  and let 
	$\Diamond\bigwedge( P_2,\neg Q_2) $ be the principal formula, where $ P=P_2,P_3 $ and $ Q=Q_3,Q_2 $. Assume $ \D $ is as 
		\begin{prooftree}
			\AXC{$ \D_1 $}
			\noLine
			\UIC{$M_1,P_2;P_1,P_3\vdash Q_3,Q_1;Q_2,N_1 $}
			\RL{R$\Box^; $.}
			\UIC{$M_1,P_1;P_2,P_3\vdash Q_3,Q_2;Q_1,N_1 $}
		\end{prooftree}
Then we have
		\begin{prooftree}
			\AXC{$ \D_1 $}
			\noLine
			\UIC{$M_1,P_2;P_1,P_3\vdash Q_3,Q_1;Q_2,N_1 $}
			\RL{IH}
			\UIC{$M'_1,M_1,P_2;P'_1,P_1,P_3\vdash Q_3,Q_1,Q'_1;Q_2,N_1,N'_1 $}
			\RL{R$\Box^; $.}
				\UIC{$M'_1,P'_1,M_1,P_1;P_2,P_3\vdash Q_3,Q_2;Q_1,N_1,Q'_1,N'_1 $}
		\end{prooftree}
	}
\end{proof}
In order to prove the admissibility of the contraction and the cut rules, we need  to state some properties in the following lemmas which are also required to prove  the admissibility of  the general versions of the rules R$\Box $ and L$\Diamond $ where the constriction of atomic for formulas has been removed.
\begin{Lem}\label{lem 3.8}
	The following rules are  admissible.
		\begin{align*}
		&\AXC{$(A\vee B)\wedge C,\Gamma\vdash\Delta$}
		\LL{{\rm (1)}}
		\UIC{$A\wedge C,\Gamma\vdash\Delta $}
		\DP
		&&
		\AXC{$\Gamma\vdash\Delta,\neg(A\vee B)\vee C$}
		\LL{{\rm (2)}}
		\UIC{$ \Gamma\vdash\Delta, \neg A\vee C $}
		\DP
		\\[0.15cm]
		&\AXC{$(A\vee B)\wedge C,\Gamma\vdash\Delta$}
		\LL{{\rm (3)}}
		\UIC{$B\wedge C,\Gamma\vdash\Delta$}
		\DP
		&&
		\AXC{$\Gamma\vdash\Delta,\neg(A\vee B)\vee C$}
		\LL{{\rm (4)}}
		\UIC{$\Gamma\vdash\Delta,\neg B\vee C $}
		\DP
		\\[0.15cm]
		&\AXC{$\neg (A\vee B)\wedge C,\Gamma\vdash\Delta$}
		\LL{{\rm (5)}}
		\UIC{$(\neg A\wedge \neg B)\wedge C,\Gamma\vdash\Delta$}
		\DP
		&&
	\AXC{$\Gamma\vdash\Delta,\neg(A\wedge B)\vee C$}
	\LL{{\rm (6)}}
	\UIC{$\Gamma\vdash\Delta,(\neg A\vee \neg B)\vee C $}
	\DP
	\\[0.15cm]
	&\AXC{$\neg (A\wedge B)\wedge C,\Gamma\vdash\Delta$}
	\LL{{\rm (7)}}
	\UIC{$\neg A\wedge C,\Gamma\vdash\Delta$}
	\DP
	&&
	\AXC{$\Gamma\vdash\Delta,(A\wedge B)\vee C$}
	\LL{{\rm (8)}}
	\UIC{$\Gamma\vdash\Delta, A\vee  C $}
	\DP
	\\[0.15cm]
	&\AXC{$\neg (A\wedge B)\wedge C,\Gamma\vdash\Delta$}
	\LL{{\rm (9)}}
	\UIC{$\neg B\wedge C,\Gamma\vdash\Delta$}
	\DP
	&&
	\AXC{$\Gamma\vdash\Delta,(A\wedge B)\vee C$}
	\LL{{\rm (10)}}
	\UIC{$\Gamma\vdash\Delta, B\vee  C $}
	\DP
     \\[0.15cm]
    &\AXC{$(A\rightarrow B)\wedge C,\Gamma\vdash\Delta$}
    \LL{{\rm (11)}}
    \UIC{$B\wedge C,\Gamma\vdash\Delta $}
    \DP
    &&
    \AXC{$\Gamma\vdash\Delta, \neg(A\rightarrow B)\vee C$}
\LL{{\rm (12)}}
\UIC{$\Gamma\vdash B\vee C,\Delta$}
\DP
\\[0.15cm]
&\AXC{$(A\rightarrow B)\wedge C,\Gamma\vdash\Delta$}
\LL{{\rm (13)}}
\UIC{$\neg A\wedge C,\Gamma\vdash\Delta $}
\DP
&&
\AXC{$\Gamma\vdash\Delta, \neg(A\rightarrow B)\vee C$}
\LL{{\rm (14)}}
\UIC{$\Gamma\vdash \neg A\vee C,\Delta$}
\DP
\\[0.15cm]
&\AXC{$\neg (A\rightarrow B)\wedge C,\Gamma\vdash\Delta$}
\LL{{\rm (15)}}
\UIC{$  (A\wedge \neg B)\wedge C,\Gamma\vdash\Delta$}
\DP
&&
\AXC{$\Gamma\vdash\Delta,(A\rightarrow B)\vee C$}
\LL{{\rm (16)}}
\UIC{$\Gamma\vdash\Delta,(\neg A\vee B)\vee C$}
\DP
	\end{align*}
	\end{Lem}
\begin{Lem}\label{ezaf1}
The following rules are  admissible.
\begin{align*}
		&\AXC{$\Diamond((A\vee B)\wedge C),\Gamma\vdash\Delta$}
		\LL{{\rm (1)}}
		\UIC{$\Diamond(A\wedge C),\Gamma\vdash\Delta $}
		\DP
		&&
		\AXC{$\Gamma\vdash\Delta,\Box(\neg(A\vee B)\vee C)$}
		\LL{{\rm (2)}}
		\UIC{$ \Gamma\vdash \Box(\neg A\vee C),\Delta $}
		\DP
		\\[0.15cm]
		&\AXC{$\Diamond((A\vee B)\wedge C),\Gamma\vdash\Delta$}
		\LL{{\rm (3)}}
		\UIC{$\Diamond(B\wedge C),\Gamma\vdash\Delta$}
		\DP
		&&\AXC{$\Gamma\vdash\Delta,\Box(\neg(A\vee B)\vee C)$}
		\LL{{\rm (4)}}
		\UIC{$\Gamma\vdash\Delta,\Box(\neg B\vee C) $}
		\DP
		\\[0.15cm]
		&\AXC{$\Diamond(\neg (A\vee B)\wedge C),\Gamma\vdash\Delta$}
		\LL{{\rm (5)}}
		\UIC{$\Diamond((\neg A\wedge \neg B)\wedge C),\Gamma\vdash\Delta$}
		\DP
		&&
		\AXC{$\Gamma\vdash \Delta,\Box(\neg(A\wedge B)\vee C)$}
		\LL{{\rm (6)}}
		\UIC{$\Gamma\vdash \Delta,\Box(\neg A\vee \neg B)\vee C) $}
		\DP
		\\[0.15cm]
		&\AXC{$\Diamond(\neg (A\wedge B)\wedge C),\Gamma\vdash\Delta$}
		\LL{{\rm (7)}}
		\UIC{$\Diamond(\neg A\wedge C,\Gamma)\vdash\Delta$}
		\DP
		&&
		\AXC{$\Gamma\vdash \Delta,\Box((A\wedge B)\vee C)$}
		\LL{{\rm (8)}}
		\UIC{$\Gamma\vdash \Delta, \Box(A\vee  C) $}
		\DP
		\\[0.15cm]
		&\AXC{$\Diamond(\neg (A\wedge B)\wedge C),\Gamma\vdash\Delta$}
		\LL{{\rm (9)}}
		\UIC{$\Diamond(\neg B\wedge C),\Gamma\vdash\Delta$}
		\DP
		&&
		\AXC{$\Gamma\vdash\Delta,\Box((A\wedge B)\vee C)$}
		\LL{{\rm (10)}}
		\UIC{$\Gamma\vdash \Delta, \Box(B\vee  C) $}
		\DP
        \\[0.15cm]
		&\AXC{$\Diamond((A\rightarrow B)\wedge C),\Gamma\vdash\Delta$}
		\LL{{\rm (11)}}
		\UIC{$\Diamond(B\wedge C),\Gamma\vdash\Delta $}
		\DP
		&&\AXC{$\Gamma\vdash \Delta,\Box(\neg(A\rightarrow B)\vee C)$}
		\LL{{\rm (12)}}
		\UIC{$\Gamma\vdash \Delta, \Box(B\vee C)$}
		\DP
		\\[0.15cm]
		&\AXC{$\Diamond((A\rightarrow B)\wedge C),\Gamma\vdash\Delta$}
		\LL{{\rm (13)}}
		\UIC{$\Diamond(\neg A\wedge C),\Gamma\vdash\Delta $}
		\DP
		&&\AXC{$\Gamma\vdash\Delta, \Box(\neg(A\rightarrow B)\vee C)$}
		\LL{{\rm (14)}}
		\UIC{$\Gamma\vdash\Delta, \Box(\neg A\vee C)$}
		\DP
		\\[0.15cm]
		&\AXC{$\Diamond(\neg (A\rightarrow B)\wedge C),\Gamma\vdash\Delta$}
		\LL{{\rm (15)}}
		\UIC{$  \Diamond(A\wedge (\neg B\wedge C)),\Gamma\vdash\Delta$}
		\DP
		&&
		\AXC{$\Gamma\vdash\Delta,\Box((A\rightarrow B)\vee C)$}
		\LL{{\rm (16)}}
		\UIC{$\Gamma\vdash\Delta,\Box(\neg A\vee B)\vee C)$}
		\DP
		\\[0.15cm]
		&\AXC{$\Diamond(\Diamond A\wedge C),\Gamma\vdash\Delta$}
		\LL{{\rm (17)}}
		\UIC{$\Diamond A,\Diamond C,\Gamma\vdash\Delta$}
		\DP
		&&
		\AXC{$\Gamma\vdash\Delta,\Box(\neg \Diamond A\vee C)$}
		\LL{{\rm (18)}}
		\UIC{$\Diamond A,\Gamma\vdash\Delta, \Box C$}
		\DP
		\\[0.15cm]
		&\AXC{$\Diamond(\Box A\wedge C),\Gamma\vdash\Delta$}
		\LL{{\rm (19)}}
		\UIC{$\Box A,\Diamond C,\Gamma\vdash\Delta$}
		\DP
		&&\AXC{$\Gamma\vdash\Delta,\Box(\neg \Box A\vee C)$}
		\LL{{\rm (20)}}
		\UIC{$\Box A,\Gamma\vdash\Delta, \Box C$}
		\DP
		\\[0.15cm]
		&\AXC{$\Diamond(\neg\Diamond A\wedge C),\Gamma\vdash\Delta$}
		\LL{{\rm (21)}}
		\UIC{$\Diamond C,\Gamma\vdash\Delta,\Diamond   A $}
		\DP
		&&
		\AXC{$\Gamma\vdash\Delta,\Box(\Diamond A\vee C)$}
		\LL{{\rm (22)}}
		\UIC{$\Gamma\vdash\Delta, \Diamond A,\Box C$}
		\DP
		\\[0.15cm]
		&\AXC{$\Diamond(\neg \Box A\wedge C),\Gamma\vdash\Delta$}
		\LL{{\rm (23)}}
		\UIC{$\Diamond C,\Gamma\vdash\Delta,\Box A$}
		\DP
		&&
		\AXC{$\Gamma\vdash\Delta,\Box(\Box A\vee C)$}
		\LL{{\rm (24)}}
		\UIC{$\Gamma\vdash\Delta, \Box A,\Box C$}
		\DP
	\end{align*}
\end{Lem}
\begin{proof}
The proof is by induction on the height of derivation of the premise in each case. As a typical example, we prove $ (1) $. 
If $\Diamond((A\vee B)\wedge C),\Gamma\vdash\Delta$
is an axiom, then
$ \Diamond((A\vee B)\wedge C) $
is not principal, and 
$ \Diamond(A\wedge C),\Gamma\vdash\Delta $
is an axiom. If  
$ \Diamond((A\vee B)\wedge C) $
in not principal, we apply the induction hypothesis to the premise and
then use the same rule to obtain deductions of
$ \Diamond(A\wedge C),\Gamma\vdash\Delta $. For example if   $ \Delta=N,Q,\Box D  $, $ \,\, \Gamma=M,P $ and the last rule  is
	  \begin{prooftree}
	  	\AXC{$ \D $}
	  	\noLine
	  	\UIC{$ \Diamond((A\vee B)\wedge C),M;P\vdash Q;N,D $}
	  	\RL{R$\Box $,}
	  	\UIC{$ \Diamond((A\vee B)\wedge C),M,P\vdash Q,N,\Box D $}
	  \end{prooftree}
then we have
	  \begin{prooftree}
	  	\AXC{$ \D $}
	  	\noLine
	  	\UIC{$ \Diamond((A\vee B)\wedge C),M;P\vdash Q;N,D $}
	  	\RL{ IH }
	  	\UIC{$ \Diamond(A\wedge C),M;P\vdash Q;N,D $}
	  	\RL{R$\Box $.}
	  	\UIC{$ \Diamond(A\wedge C),M,P\vdash Q,N,\Box D $}
	  \end{prooftree}
If on the other hand 
$ \Diamond((A\vee B)\wedge C) $
is principal, the last  rule is: 
	\begin{prooftree}
		\AXC{$ \D $}
		\noLine
		\UIC{$ (A\vee B)\wedge C,M;P\vdash Q;N $}
		\RL{L$\Diamond $,}
		\UIC{$ \Diamond((A\vee B)\wedge C),M,P\vdash Q,N $}
		\end{prooftree}
then we have
	\begin{prooftree}
		\AXC{$ \D $}
		\noLine
		\UIC{$ (A\vee B)\wedge C,M;P\vdash Q;N$ }
		\RL{Lemma \ref{lem 3.8}}
		\UIC{$ A\wedge C,M;P\vdash Q;N $}
		\RL{L$\Diamond $.}
		\UIC{$ \Diamond(A\wedge C),M,P\vdash Q,N $}
		\end{prooftree}
\end{proof}
\begin{Lem}\label{*}
The following rules are admissible.
	\begin{center}
		\AXC{$ \Diamond A,\Gamma\vdash \Delta $}
		\LL{{\rm (\rom{1})}}
		\UIC{$ A,\Gamma\vdash \Delta $}
		\DP
		$ \qquad $
		\AXC{$ \Gamma\vdash\Delta,\Box A $}
		\LL{\rm (\rom{2})}
		\UIC{$\Gamma\vdash\Delta, A $}
		\DP
	\end{center}
\end{Lem}
Before we prove this lemma, we need to state some properties in the following.

Similar to the  propositional logic, every formula $ A $  has an  equivalent disjunctive normal form (DNF) and an  equivalent conjunctive  normal  form (CNF), such that each conjunction in DNF and disjunction  in CNF   are respectively as follows:
$$ (p_1\wedge \cdots \wedge p_n\wedge \neg q_1\wedge\cdots\wedge\neg q_m\wedge B_1\wedge\cdots\wedge B_k\wedge \neg C_1\wedge\cdots\wedge\neg C_l) $$
and
$$ (p_1\vee \cdots \vee p_n\vee \neg q_1\vee\cdots\vee\neg q_m\vee B_1\vee\cdots\vee B_k\vee \neg C_1\vee\cdots\vee\neg C_l) $$
where $ p_1,\ldots,p_n $ and $ q_1,\ldots,q_m $ are atomic, and $ B_1,\ldots,B_k $ and $ C_1\ldots, C_l $ are modal formulas.

For these conjunctions and   disjunctions we have the following lemma.
\begin{Lem}\label{NF}
Let $ A $ be a formula, and let
$$ (p_1\wedge \cdots \wedge p_n\wedge \neg q_1\wedge\cdots\wedge\neg q_m\wedge B_1\wedge\cdots\wedge B_k\wedge \neg C_1\wedge\cdots\wedge\neg C_l) $$
be a  conjunction in the DNF 
and
$$ (p_1\vee \cdots \vee p_n\vee \neg q_1\vee\cdots\vee\neg q_m\vee B_1\vee\cdots\vee B_k\vee \neg C_1\vee\cdots\vee\neg C_l) $$
be a  disjunction in the CNF,  where $ p_1,\ldots,p_n $ and $ q_1,\ldots,q_m $ are atomic, and $ B_1,\ldots,B_k $ and $ C_1\ldots, C_l $ are modal formulas. Then  the following rules are admissible.
\[
   \begin{array}{l}
	\AXC{$ \Diamond A,\Gamma\vdash\Delta $}
	\LL{\rm (i)}
	\UIC{$B_1,\ldots,B_k, \Diamond (p_1\wedge\cdots \wedge p_n\wedge  \neg q_1\wedge\cdots\wedge\neg q_m),\Gamma\vdash\Delta, C_1\ldots, C_l $}
	\DP\\
	$ \quad $\\
	\AXC{$ \Gamma\vdash\Delta,\Box A $}
	\LL{\rm (ii)}
	\UIC{$ C_1\ldots, C_l,\Gamma\vdash\Delta,B_1,\ldots,B_k, \Box (p_1\vee\cdots \vee p_n\vee  \neg q_1\vee\cdots\vee\neg q_m) $}
	\DP\\
	\end{array}
	\]
\end{Lem}
\begin{proof}
This  easily follows from Lemma \ref{ezaf1}.
\end{proof}
\begin{Lem}\label{general;}
The following rule is  admissible:
	\begin{prooftree}
		\AXC{$ \Gamma;P\vdash Q;\Delta $}
		\UIC{$ \Gamma,P\vdash Q,\Delta $}
	\end{prooftree}
where $ \Gamma $ and $ \Delta $ are multisets of arbitrary formulas, and $ P$ and $Q $  are multisets of atomic formulas.
\end{Lem}
\begin{proof}
The proof is by induction on the height of derivation of the premise. 
If  the premise
is an axiom, then  the conclusion is an axiom. 
For the induction step, we consider only  the cases where  the last rule  is L$\Diamond $, since the rule R$\Box $ is treated symmetrically and  for the remaining 
rules it  suffices to apply the induction hypothesis to the premise and
then use the same rule to obtain deduction of
	$ \Gamma,P\vdash Q,\Delta $.
	
Case 1. Let $ \Gamma=\Diamond A,M,P_1 $ and 
	$ \Delta=Q_1,N $ 
and let $ \Diamond A $ be the principal formula.
	\begin{prooftree}
		\AXC{$   A,M;P_1,P\vdash_n Q,Q_1;N $}
		\RL{L$\Diamond $.}
		\UIC{$ \Diamond A,M,P_1;P\vdash_{n+1}Q;Q_1,N $}
	\end{prooftree}
We have
	\begin{prooftree}
		\AXC{$   A,M;P_1,P\vdash_n Q,Q_1;N $}
		\RL{L$\Diamond $.}
		\UIC{$  \Diamond A,M,P_1,P\vdash_{n+1}Q,Q_1,N $}
	\end{prooftree}

Case 2. Let $ \Gamma=M,P_1 $ and  $ \Delta=Q_1,N $, and let $\Diamond\bigwedge( P_2,  \neg Q_2) $ be the principal formula, where $ P=P_2,P_3 $ and $ Q=Q_3,Q_2 $.
\begin{prooftree}
	\AXC{$ M,P_2;P_1,P_3\vdash_n Q_3,Q_1;Q_2,N $}
	\RL{L$\Diamond^; $.}
	\UIC{$ M,P_1;P_2,P_3\vdash_{n+1}Q_3,Q_2;Q_1,N $}
\end{prooftree}
In this case, by induction hypothesis on the premise  we are done.
\end{proof}
The following corollary which is a special case of Lemma \ref{*} and  is used in its proof can be concluded from Lemma \ref{general;}.  
\begin{Cor}\label{cor ;}
	The following rules are admissible
	\begin{prooftree}
		\AXC{$ \Diamond (p_1\wedge\cdots \wedge p_n\wedge  \neg q_1\wedge\cdots\wedge\neg q_m),\Gamma\vdash\Delta $}
		\LL{\rm (i)}
		\UIC{$ p_1,\ldots, p_n, \Gamma\vdash\Delta,q_1,\ldots, q_m$}
	\end{prooftree}
	\begin{prooftree}
		\AXC{$ \Gamma\vdash\Delta, \Box (p_1\vee\cdots \vee p_n\vee  \neg q_1\vee\cdots\vee\neg q_m) $}
		\LL{\rm (ii)}
		\UIC{$q_1,\ldots, q_m, \Gamma\vdash\Delta,  p_1,\ldots, p_n $}
	\end{prooftree}
\end{Cor}
\begin{Lem}\label{constA}
	$ \, $
	\begin{itemize}
		\item[{\rm (i)}]
		If
		$B_1,\ldots,B_k, p_1,\ldots, p_n, \Gamma\vdash\Delta,q_1,\ldots, q_m, C_1,\ldots, C_l $, for each conjunction
		\[(p_1\wedge\cdots \wedge p_n\wedge \neg q_1\wedge\cdots\wedge \neg q_m \wedge B_1\wedge\cdots\wedge B_k\wedge \neg C_1\wedge \cdots\wedge \neg C_l) \]
		in the DNF of $ A $, then $ A,\Gamma\vdash\Delta $.
		\item[{\rm (ii)}]
		If 
		$C_1\ldots, C_l,q_1,\ldots, q_m, \Gamma\vdash\Delta, B_1,\ldots,B_k, p_1,\ldots, p_n $, for each disjunction
		\[( p_1\vee\cdots \vee p_n\vee \neg q_1\vee\cdots\vee \neg q_m \vee B_1\vee\cdots\vee B_k\vee \neg C_1\vee \cdots\vee \neg C_l) \]
		in the CNF of $ A $, then $ \Gamma\vdash\Delta,A $.
	\end{itemize}
\end{Lem}
We now deduce Lemma \ref{*} from Lemma \ref{NF}, Corollary \ref{cor ;} and Lemma \ref{constA}.
\begin{proof}[ {\rm \textbf{Proof of Lemma \ref{*}}}]
	For the proof of  part $ {\rm (\rom{1})} $, we have  
\begin{prooftree}
	\AXC{$ \Diamond A,\Gamma\vdash\Delta $}
	\RL{Lemma \ref{NF}}
	\UIC{$B_1,\ldots,B_k, \Diamond (p_1\wedge\cdots \wedge p_n\wedge  \neg q_1\wedge\cdots\wedge\neg q_m),\Gamma\vdash\Delta, C_1\ldots, C_l $}
	\RL{Corollary \ref{cor ;},}
	\UIC{$B_1,\ldots,B_k, p_1,\ldots, p_n, \Gamma\vdash\Delta,q_1,\ldots, q_m, C_1,\ldots, C_l $}
	\end{prooftree}
for every conjunction 
$ (p_1\wedge \cdots \wedge p_n\wedge \neg q_1\wedge\cdots\wedge\neg q_m\wedge B_1\wedge\cdots\wedge B_k\wedge \neg C_1\wedge\cdots\wedge\neg C_l) $
 in the DNF of $ A $.
	Now the proof will be concluded  by applying Lemma  \ref{constA}.
	
	Similarly, for the proof of   part $ {\rm (\rom{2})} $, we have
\begin{prooftree}
	\AXC{$ \Gamma\vdash\Delta,\Box A $}
	\RL{Lemma \ref{NF}}
	\UIC{$ C_1\ldots, C_l,\Gamma\vdash\Delta,B_1,\ldots,B_k, \Box (p_1\vee\cdots \vee p_n\vee  \neg q_1\vee\cdots\vee\neg q_m) $}
	\RL{Corollary \ref{cor ;},}
	\UIC{$C_1\ldots, C_l,q_1,\ldots, q_m, \Gamma\vdash\Delta, B_1,\ldots,B_k, p_1,\ldots, p_n $}
\end{prooftree}
for every  disjunction
$ (p_1\vee \cdots \vee p_n\vee \neg q_1\vee\cdots\vee\neg q_m\vee B_1\vee\cdots\vee B_k\vee \neg C_1\vee\cdots\vee\neg C_l) $ in the CNF of $ A $.
Again we are done by applying Lemma  \ref{constA}.
\end{proof}
\begin{Cor}\label{**}
	The following rule is admissible.
	\begin{prooftree}
		\AXC{$ \Diamond\Gamma_1,\Gamma;P\vdash Q ;\Delta,\Box\Delta_1 $}
		\RL{,}
		\UIC{$ \Gamma_1,\Gamma;P\vdash Q ;\Delta,\Delta_1 $}
	\end{prooftree}
	where
	$ \Gamma_1 $ and 
	$ \Delta_1 $ are multisets of arbitrary formulas.
\end{Cor}
In  Lemma   \ref{*} and Corollary \ref{**},  the deduction  of the conclusion sequent is produced by permutations of the rules in the deduction of the premise sequent.  The following example  is provided to show this permutation.
\begin{Examp}\label{ex1 for 3.13}
	{\rm 
		Let $ \D $ be a derivation of 
		$ M,P,\Diamond\Diamond r, \Diamond(p\rightarrow\Box \Diamond q)\vdash \Box(\Diamond s\vee t),Q,N $.
		We get a derivation $ \D' $ for
		\[ M,P,\Diamond r, \Diamond(p\rightarrow\Box \Diamond q)\vdash \Diamond s\vee t,Q,N. \]
		
		Suppose  $ \D $ is as 
		\begin{prooftree}
			\AXC{$ \D_1 $}
			\noLine
			\UIC{$ M,\Diamond\Diamond r;P\vdash Q,p;\Diamond s \vee t,N $}
			\RL{R$\Box $}
			\UIC{$ M,\Diamond\Diamond r;P\vdash Q;\Box(\Diamond s\vee t) ,p,N $}
			\AXC{$ \D_2 $}
			\noLine
			\UIC{$ M,\Diamond q ,\Diamond\Diamond r;P\vdash Q;\Diamond s\vee t,N$}
			\RL{R$\Box $}
			\UIC{$ M,\Diamond q ,\Diamond\Diamond r;P\vdash Q;\Box(\Diamond s\vee t),N$}
			\RL{L$\Box $}
			\UIC{$ M,\Box \Diamond q,\Diamond\Diamond r;P\vdash Q;\Box(\Diamond s\vee t),N$}  
			\RL{L$\rightarrow $}
			\BIC{$ M,\Diamond\Diamond r, p\rightarrow\Box \Diamond q;P\vdash Q;\Box(\Diamond s\vee t),N $}
			\RL{L$\Diamond $,W}
			\UIC{$ M,P,\Diamond\Diamond r, \Diamond(p\rightarrow\Box \Diamond q)\vdash \Box(\Diamond s\vee t),Q,N $}
		\end{prooftree}
		where 
		$ \D_1 $ and $ \D_2 $ are respectively
		as follows:
		\begin{prooftree}
			\AXC{$ \D_{1_1} $}
			\noLine
			\UIC{$ M,\Diamond r;P\vdash Q;\Diamond s,p,N  $}
			\RL{L$ \Diamond^; $}
			\UIC{$ M,\Diamond r;P\vdash Q,p;\Diamond s,N  $}
			\RL{L$\Diamond $}
			\UIC{$ M,\Diamond\Diamond r;P\vdash Q,p;\Diamond s ,N $}
			\AXC{$ \D_{1_2} $}
			\noLine
			\UIC{$ M,\Diamond r;P\vdash Q,t;p,N $}
			\RL{L$ \Diamond^; $}
			\UIC{$ M,\Diamond r;P\vdash Q,p,t;N $}
			\RL{L$\Diamond $}
			\UIC{$  M,\Diamond\Diamond r;P\vdash Q,p;t,N $}
			\RL{R$\vee $}
			\BIC{$  M,\Diamond\Diamond r;P\vdash Q,p;\Diamond s \vee t,N $}
		\end{prooftree}
		\begin{prooftree}
			\AXC{$ \D_{2_1} $}
			\noLine
			\UIC{$ M,\Diamond q,\Diamond r;P\vdash Q;\Diamond s,N$}
			\RL{L$\Diamond $}
			\UIC{$ M,\Diamond q,\Diamond\Diamond r;P\vdash Q;\Diamond s,N$}
			\AXC{$ \D_{2_2} $}
			\noLine
			\UIC{$ M,\Diamond q ,\Diamond r;P\vdash Q,t;N$}
			\RL{L$\Diamond $}
			\UIC{$ M,\Diamond q ,\Diamond\Diamond r;P\vdash Q; t,N$}
			\RL{R$\vee $}
			\BIC{$ M,\Diamond q ,\Diamond\Diamond r;P\vdash Q;\Diamond s\vee t,N$}
		\end{prooftree}
		Thus by permutation of the rules L$ \rightarrow $ and the rules L$ \Diamond $  and R$ \vee $  we get  $ \D' $ as 
		\begin{prooftree}
			\AXC{$ \D_{1_1} $}
			\noLine
			\UIC{$ M,\Diamond r;P\vdash Q;\Diamond s,p,N  $}
			\AXC{$ \D_{2_1} $}
			\noLine
			\UIC{$ M,\Diamond q,\Diamond r;P\vdash Q;\Diamond s,N$}
			\RL{L$\rightarrow $}
			\BIC{$ M, \Diamond r,p\rightarrow\Diamond q;P\vdash Q;\Diamond s ,N  $}
			\RL{L$\Diamond $}
			\UIC{$ M, \Diamond r,\Diamond(p\rightarrow\Diamond q),P\vdash Q,\Diamond s,N   $}
		\end{prooftree}
		and 
		\begin{prooftree}
			\AXC{$ \D_{1_2} $}
			\noLine
			\UIC{$ M,\Diamond r;P\vdash Q,t;p,N   $}
			\AXC{$ \D_{2_2} $}
			\noLine
			\UIC{$ M,\Diamond q ,\Diamond r;P\vdash Q,t;N $}
			\RL{L$\rightarrow $}
			\BIC{$  M,\Diamond r,p\rightarrow\Diamond q;P\vdash Q,t;N   $}
			\RL{L$\Diamond $}
			\UIC{$ M, \Diamond r,\Diamond(p\rightarrow\Diamond q), P\vdash Q,t ,N  $}
		\end{prooftree}
		and then by applying the rule R$\vee $ to derive 
		$ M,P,	\Diamond r,\Diamond(p\rightarrow\Diamond q)\vdash \Diamond s\vee t,Q,N $.
	}
\end{Examp}
 We now  prove   the admissibility of the contraction rule that is required for the proof of the  admissibility of the cut rule (Theorem \ref{cut}).
\begin{Lem} \label{C}
	The rules of contraction,
	\begin{align*}
		&\AXC{$ \Gamma;p,p,P\vdash_n Q;\Delta  $}
		\RL{$ \text{\rm LC}^{;} $}
		\UIC{$ \Gamma;p,P\vdash_n Q;\Delta $}
		\DP
	&&
		\AXC{$ \Gamma;P\vdash_n Q,q,q;\Delta  $}
	\RL{$ \text{\rm RC}^{;} $}
	\UIC{$ \Gamma;P\vdash_n Q,q;\Delta $}
	\DP
		\\[0.4em]		
&
\AXC{$ A,A,\Gamma;P\vdash Q;\Delta  $}
\RL{LC}
\UIC{$ A,\Gamma;P\vdash Q;\Delta $}
\DP
&&
\AXC{$ \Gamma;P\vdash Q;\Delta,A,A  $}
\RL{RC,}
\UIC{$ \Gamma;P\vdash Q;\Delta,A $}
\DP
\end{align*}	
	are  admissible, where $ A $ is an arbitrary formula,   $ p $  and $ q $ are atomic formulas, and both $p$'s in the rule $ \text{\rm LC}^{;} $ as well as both $q$'s in the rule $ \text{\rm RC}^{;} $  have the same corresponding formulas $ \Diamond\bigwedge X_i $ or $ \Box\bigvee Y_j $  as in Subsection \ref{Another version}.
\end{Lem}
\begin{proof}
	{\rm
		All cases are    proved simultaneously by induction on complexity of $ A $ with subinduction on the height of  derivation of the premises. We only consider some cases that  $ p $ is in the principal formula or $ A $  is a principal formula, the other cases are proved by a similar argument.
		In all cases, if the premise is an axiom, then the conclusion is an axiom too.
		
For the rule $ \text{LC}^{;} $,  let $ \Gamma=M_1,P_1$ and $ \Delta=Q_1,N_1$, and let $ P=P_2,P_3 $ and $  Q=Q_3,Q_2  $. If  the last rule  is 
\begin{prooftree}
	\AXC{$ \D $}
	\noLine
	\UIC{$ M_1,p,p,P_2;P_1,P_3\vdash_n Q_3,Q_1;Q_2,N_1 $}
	\RL{$R $,}
	\UIC{$M_1,P_1;p,p,P_2,P_3\vdash_{n+1} Q_3,Q_2;Q_1,N_1 $}
\end{prooftree}
where $ R $ is R$\Box^; $  or  L$\Diamond^; $ with principal formula $ \Diamond\bigwedge(p, p, P_2,\neg Q_2) $ or $ \Box\bigvee(\neg p,\neg p, \neg P_2, Q_2) $,
then the conclusion is obtained by applying induction hypothesis to the middle part which preserves height then  to the first part, and then by applying the rule $ R$ on the formula $ \Diamond\bigwedge(p,  P_2,\neg Q_2) $ or $ \Box\bigvee(\neg p, \neg P_2, Q_2) $:
\begin{prooftree}
	\AXC{$ \D $}
	\noLine
	\UIC{$ M_1,p,p,P_2;P_1,P_3\vdash_n Q_3,Q_1;Q_2,N_1 $}
	\RL{IH }
	\UIC{$ M_1,p,P_2;P_1,P_3\vdash_n  Q_3,Q_1;Q_2,N_1  $}
	\RL{$R $.}
	\UIC{$M_1,P_1;p,P_2,P_3\vdash Q_3,Q_2;Q_1,N_1 $}
\end{prooftree}

For the rule L$C $, let $ \D $ be a derivation  of
$ A,A,\Gamma;P\vdash Q;\Delta  $, and  let $ \Gamma=M_1,P_1$ and $ \Delta=Q_1,N_1$. There are some cases according to the complexity of $ A $.\\
Case 1. $ A= \Diamond B $ and the last rule is L$\Diamond $:
		\begin{prooftree}
			\AXC{$ \D' $}
			\noLine
			\UIC{$ B,\Diamond B,M_1;P_1,P\vdash Q,Q_1;N_1 $}
			\RL{L$\Diamond $,}
			\UIC{$\Diamond B,\Diamond B,M_1,P_1;P\vdash Q;Q_1,N_1 $}
		\end{prooftree}
		 we have
		\begin{prooftree}
			\AXC{$ \D' $}
			\noLine
			\UIC{$ B,\Diamond B,M_1;P_1,P\vdash Q,Q_1;N_1 $}
			\RL{ Lemma \ref{*} }
			\UIC{$ B, B,M_1;P_1,P\vdash Q,Q_1;N_1 $}
			\RL{ IH }
			\UIC{$B,M_1;P_1,P\vdash Q,Q_1;N_1 $}
			\RL{L$\Diamond $.}
		\UIC{$ \Diamond B,M_1,P_1;P\vdash Q;Q_1,N_1 $}
		\end{prooftree}
Case 2. $ A=\Box B$ and the last rule is	L$\Box $:
\begin{prooftree}
	\AXC{$ \D' $}
	\noLine
	\UIC{$ B,\Box B,\Box B,\Gamma;P\vdash Q;\Delta $}
	\RL{L$\Box $,}
	\UIC{$\Box B,\Box B,\Gamma;P\vdash Q;\Delta $}
\end{prooftree}
we have
\begin{prooftree}
	\AXC{$ \D' $}
	\noLine
	\UIC{$ B,\Box B,\Box B,\Gamma;P\vdash Q;\Delta $}
	\RL{IH}
	\UIC{$ B,\Box B,\Gamma;P\vdash Q;\Delta $}
	\RL{L$\Box $.}
	\UIC{$\Box B,\Gamma;P\vdash Q;\Delta $}
\end{prooftree}
Case 3. $ A=B\rightarrow C $ and the last rule is	L$\rightarrow $:
\begin{prooftree}
	\AXC{$ \D_1 $}
	\noLine
	\UIC{$ B\rightarrow C,\Gamma;P\vdash Q;\Delta,B $}
	\AXC{$ \D_2 $}
	\noLine
	\UIC{$ C,B\rightarrow C,\Gamma;P\vdash Q;\Delta $}
	\RL{L$\rightarrow $,}
	\BIC{$B\rightarrow C,B\rightarrow C,\Gamma;P\vdash Q;\Delta $}
\end{prooftree}		
By the inversion lemma applied to the first premise, $ \Gamma;P\vdash Q;\Delta,B,B $, and
applied to the second premise, $ C,C,\Gamma;P\vdash Q;\Delta $. We then use the induction hypothesis and obtain\\ $ \Gamma;P\vdash Q;\Delta,B $, A and $ C,\Gamma;P\vdash Q;\Delta $. Thus by L$\rightarrow $, we have $ B\rightarrow C,\Gamma;P\vdash Q;\Delta $.	\\
The other cases are proved by a similar argument.
	}
\end{proof}
In the rest of this section, we prove the admissibility of the general versions of the rules R$\Box $ and L$\Diamond $  that is also required for the proof of the  admissibility of the cut rule.
\begin{Lem}\label{General}
	The general versions of the rules R$\Box $  and L$\Diamond $,
	\begin{center}
		\AXC{$ M;\Gamma,P\vdash Q,\Delta;N, A $}
		\RL{R$\Box^G $}
		\UIC{$M,\Gamma;P\vdash Q;\Delta,N,\Box A $}
		\DP
		$ \quad $
		\AXC{$A,M;\Gamma,P\vdash Q,\Delta;N  $}
		\RL{L$\Diamond^G $}
		\UIC{$ \Diamond A,M,\Gamma;P\vdash Q;\Delta,N $}
		\DP
	\end{center}
	are admissible, where $ \Gamma$ and  $\Delta $ are multisets of arbitrary formulas.
\end{Lem}
Before we prove this lemma, we need  to state the  following lemmas. 
\begin{Lem}\label{sadeh}
	The following rules are admissible, where  the  formulas $ A\circ B $, $ \circ\in \{\wedge,\vee,\rightarrow\} $, in the premises, and  $ A $ and $ B $   in the conclusions of the rules have the same related formulas.
	\begin{align*}
	&\AXC{$ \Gamma;A\wedge B,P\vdash Q;\Delta $}
	\LL{{\rm (1)}}
	\UIC{$ \Gamma;A,B,P\vdash Q;\Delta $}
	\DP
	&&
	\AXC{$ \Gamma;P\vdash Q,A\vee B;\Delta $}
	\LL{{\rm (2)}}
	\UIC{$ \Gamma;P\vdash Q,A,B;\Delta $}
	\DP
	 \\[0.15cm]
	&\AXC{$ \Gamma;A\vee B,P\vdash Q;\Delta $}
	\LL{{\rm (3)}}
	\UIC{$ \Gamma;A,P\vdash Q;\Delta $}
	\DP
	&&
	\AXC{$ \Gamma;P\vdash Q,A\wedge B;\Delta $}
	\LL{{\rm (4)}}
	\UIC{$ \Gamma;P\vdash Q,A;\Delta $}
	\DP
	\\
	&\AXC{$ \Gamma;A\vee B,P\vdash Q;\Delta $}
	\LL{{\rm (5)}}
	\UIC{$ \Gamma;B,P\vdash Q;\Delta $}
	\DP
	&&
	\AXC{$ \Gamma;P\vdash Q,A\wedge B;\Delta $}
	\LL{{\rm (6)}}
	\UIC{$ \Gamma;P\vdash Q,B;\Delta $}
	\DP
	\\[0.15cm]
	&\AXC{$ \Gamma;A\rightarrow B,P\vdash Q;\Delta $}
	\LL{{\rm (7)}}
	\UIC{$ \Gamma;B,P\vdash Q;\Delta $}
	\DP
	&&
	\AXC{$ \Gamma;P\vdash Q,A\rightarrow  B;\Delta $}
	\LL{{\rm (8)}}
	\UIC{$ \Gamma;A,P\vdash Q,B;\Delta $}
	\DP
	\\[0.15cm]
	&\AXC{$\Gamma;A\rightarrow B,P\vdash Q;\Delta  $}
	\LL{{\rm (9)}}
	\UIC{$ \Gamma; P\vdash Q, A;\Delta $}
	\DP
	&&
	\AXC{$ \Gamma;P\vdash Q,\neg A;\Delta $}
	\LL{{\rm (10)}}
	\UIC{$ \Gamma; A,P\vdash Q ;\Delta $}
	\DP
	\\[0.15cm]
	&\AXC{$ \Gamma;\neg A,P\vdash Q;\Delta $}
	\LL{{\rm (11)}}
	\UIC{$ \Gamma;  P\vdash Q, A;\Delta $}
	\DP
	&&
	\end{align*}
\end{Lem}
The following lemma  state that, the rules of the above lemma are invertible.
\begin{Lem}\label{sadeh.inv}
	The following rules are admissible,   the  formulas  $ A $ and $ B $ in the premises, and $ A\circ B $, $ \circ\in \{\wedge,\vee,\rightarrow\} $, in the conclusions of the rules have the same related formulas.
	\begin{align*}
	&\AXC{$ \Gamma;A,B,\Gamma'\vdash \Delta';\Delta $}
	\LL{{\rm (1)}}
	\UIC{$ \Gamma;A\wedge B,\Gamma'\vdash \Delta';\Delta $}
	\DP
	&&
	\AXC{$ \Gamma;\Gamma'\vdash \Delta',A;\Delta $}
	\AXC{$ \Gamma;\Gamma'\vdash \Delta',B;\Delta $}
	\LL{{\rm (2)}}
	\BIC{$ \Gamma;\Gamma'\vdash \Delta',A\wedge B;\Delta $} 
	\DP
	\\[0.15cm]
	&\AXC{$ \Gamma;A,\Gamma'\vdash \Delta';\Delta $}
	\AXC{$ \Gamma;B,\Gamma'\vdash \Delta';\Delta $}
	\LL{{\rm (3)}}
	\BIC{$ \Gamma;A\vee B,\Gamma'\vdash \Delta';\Delta $}
	\DP
	&&
	\AXC{$ \Gamma;\Gamma'\vdash \Delta',A,B;\Delta $}
	\LL{{\rm (4)}}
	\UIC{$ \Gamma;\Gamma'\vdash \Delta',A\vee B;\Delta $}
	\DP
	\\[0.15cm]
	&\AXC{$ \Gamma; \Gamma'\vdash \Delta', A;\Delta $}
	\AXC{$ \Gamma;B,\Gamma'\vdash \Delta';\Delta $}
	\LL{{\rm (5)}}
	\BIC{$ \Gamma;A\rightarrow B,\Gamma\vdash \Delta';\Delta $}    
	\DP
	&&
	\AXC{$ \Gamma;A,\Gamma'\vdash \Delta',B;\Delta $}
	\LL{{\rm (6)}}
	\UIC{ $ \Gamma;\Gamma'\vdash \Delta',A\rightarrow  B;\Delta $}  
	\DP
	\\[0.15cm]
	&\AXC{$ \Gamma;\Gamma'\vdash \Delta', A;\Delta $}
	\LL{{\rm (7)}}
	\UIC{$ \Gamma; \neg A,\Gamma'\vdash \Delta' ;\Delta $}
	\DP
	&&
	\AXC{$ \Gamma; A,\Gamma'\vdash \Delta';\Delta $}
	\LL{{\rm (8)}}
	\UIC{$ \Gamma;  \Gamma'\vdash \Delta', \neg A;\Delta $}
	\DP
	\end{align*}
\end{Lem}
\begin{proof}
The proof is by  induction on the height of  derivation in each case.	
\end{proof}
\begin{Lem} \label{lem3.7}
	The following rules are  admissible.
	\begin{align*}
	&\AXC{$ \Gamma,\Diamond A;P\vdash Q;\Delta $}
	\LL{{\rm (1)}}
	\UIC{$ \Gamma;A,P\vdash Q;\Delta $}
	\DP
	&&
	\AXC{$ \Gamma;P\vdash Q;\Box A,\Delta $}
	\LL{{\rm (2)}}
	\UIC{$ \Gamma;P\vdash Q,A;\Delta $}
	\DP
	\\[0.15cm]
	&\AXC{$ \Gamma;\Diamond A,P\vdash Q;\Delta $}
	\LL{{\rm (3)}}
	\UIC{$ \Gamma,\Diamond A;P\vdash Q;\Delta $}
	\DP
	&&
	\AXC{$ \Gamma;P\vdash Q,\Diamond A;\Delta $}
	\LL{{\rm (4)}}
	\UIC{$ \Gamma;P\vdash Q;\Diamond A,\Delta $}
	\DP
	\\[0.15cm]
	&\AXC{$ \Gamma;\Box A,P\vdash Q;\Delta $}
	\LL{{\rm (5)}}
	\UIC{$ \Gamma,\Box A;P\vdash Q;\Delta $}
	\DP
	&&
	\AXC{$ \Gamma;P\vdash Q,\Box A;\Delta $}
	\LL{{\rm (6)}}
	\UIC{$ \Gamma;P\vdash Q;\Box A,\Delta $}
	\DP
	\end{align*}
\end{Lem}
Actually Lemma \ref{ezaf1} is an expression of Lemmas  \ref{sadeh} and  \ref{lem3.7}  in the  original notation, and Lemma \ref{lem 3.8} is its correspondence  without modals.   

We now deduce Lemma \ref{General} from Lemmas \ref{sadeh} and  \ref{lem3.7}.
\begin{proof}[{\rm \textbf{Proof of Lemma \ref{General}}}]
	The proof is by  induction on complexity of the formulas in
	$ \Gamma$ and $\Delta $. We prove the admissibility of  the rule
	R$\Box $, the rule
	L$\Diamond $
	is  treated symmetrically.
	
	Case 1. Let
	$ \Gamma=\Gamma_1,B\wedge C $. We have
	\begin{prooftree}
		\AXC{$ M;\Gamma_1,B\wedge C,P\vdash Q,\Delta;N, A $}
		\RL{Lemma \ref{sadeh} }
		\UIC{$ M;\Gamma_1,B, C,P\vdash Q,\Delta;N, A $}
		\RL{IH}
		\UIC{$ M,\Gamma_1,B, C;P\vdash Q;\Delta,N, \Box A $}
		\RL{L$\wedge.$}
		\UIC{$ M,\Gamma_1,B\wedge C;P\vdash Q;\Delta,N, \Box A $}
	\end{prooftree}
	
	Case 2. Let
	$ \Gamma=\Gamma_1,B\vee C $. We have
	\begin{prooftree}
		\AXC{$ M;\Gamma_1,B\vee C,P\vdash Q,\Delta;N, A $}
		\RL{ \ref{sadeh}}
		\UIC{$ M;\Gamma_1,B ,P\vdash Q,\Delta;N, A $}
		\RL{IH,}
		\UIC{$ M,\Gamma_1,B ;P\vdash Q;\Delta,N, \Box A $}
		\AXC{$M;\Gamma_1,B\vee C,P\vdash Q,\Delta;N, A $}
		\RL{ \ref{sadeh}}
		\UIC{$ M;\Gamma_1,C,P\vdash Q,\Delta;N, A $}
		\RL{IH,}
		\UIC{$ M,\Gamma_1,C;P\vdash Q;\Delta,N, \Box A $}
		\RL{L$\vee $.}
		\BIC{$ M,\Gamma_1,B\vee C;P\vdash Q; \Delta,N,\Box A. $}
	\end{prooftree}
	Case 3.	Let
	$ \Gamma=\Gamma_1,B\rightarrow C $. We have
	\begin{prooftree}
		\AXC{$M;\Gamma_1,B\rightarrow C,P\vdash Q,\Delta;N, A $}
		\RL{ \ref{sadeh}}
		\UIC{$ M;\Gamma_1,C ,P\vdash Q,\Delta;N, A$}
		\RL{IH,}
		\UIC{$ M,\Gamma_1,C ;P\vdash Q;\Delta,N, \Box A $}
		\AXC{$M;\Gamma_1,B\rightarrow C,P\vdash Q,\Delta;N, A $}
		\RL{ \ref{sadeh}}
		\UIC{$ M;\Gamma_1  ,P\vdash Q,\Delta, B;N, A$}
		\RL{IH,}
		\UIC{$ M,\Gamma_1  ;P\vdash Q ;B,\Delta,N, \Box A$}
		\RL{L$\rightarrow $.}
		\BIC{$ M,\Gamma_1,B\rightarrow C;\Gamma_2\vdash Q; \Delta,N,\Box A. $}
	\end{prooftree}
	Case 4. Let $ \Gamma=\Gamma_1,\Diamond B $. We have
	\begin{prooftree}
		\AXC{$ M;\Gamma_1,\Diamond B,P\vdash  Q,\Delta;N,A $}
		\RL{Lemma \ref{lem3.7}}
		\UIC{$ M,\Diamond B;\Gamma_1,P\vdash  Q,\Delta;N,A $}
		\RL{IH.}
		\UIC{$ M,\Diamond B,\Gamma_1;P\vdash  Q;\Delta,N,A $}
	\end{prooftree}
	The other cases are proved similarly.
\end{proof}
In the above lemma, similar to   Example \ref{ex1 for 3.13} the deduction  of the conclusion sequent is produced by permutations of the rules in the deduction of the premise sequent.  The following example  is provided to show this permutation.
\begin{Examp}\label{ex1 for 3.12}
	{\rm
	We show that  the following  rule  is admissible.
	\begin{prooftree}
		\AXC{$M,\Diamond((r\wedge\Box s)\wedge (\neg (\Box r\rightarrow s))) \vdash p\wedge (q\vee\Box s), N $}
		\RL{R$ \Box^G $}
		\UIC{$  M,r\wedge\Box s\vdash\Box r\rightarrow s, \Box(p\wedge (q\vee\Box s)),N $}
		\end{prooftree}
where the premise is derived  by $ \D $ as follows:
\begin{prooftree}
	\AXC{$ \D_1 $}
	\noLine
	\UIC{$M,\Box s,\Box r ;r \vdash s; p ,N $}
	\RL{R$ \Box^; $}
	\UIC{$M,r,\Box s,\Box r \vdash p; s ,N $}
	\RL{L$\wedge $}
	\UIC{$ M,r\wedge\Box s,\Box r \vdash p; s ,N $}
	\RL{R$\rightarrow $}
	\UIC{$ M,r\wedge\Box s \vdash p;\Box r\rightarrow s ,N$}
	\RL{L$\neg  $}
	\UIC{$M,r\wedge\Box s, \neg( \Box r\rightarrow s )\vdash p;N $}
	\RL{L$\wedge  $}
	\UIC{$ M,(r\wedge\Box s)\wedge (\neg( \Box r\rightarrow s ))\vdash p;N$}
	\RL{L$\Diamond $}
	\UIC{$ M,\Diamond((r\wedge\Box s)\wedge (\neg (\Box r\rightarrow s))) \vdash p, N  $}
	      \AXC{$ \D_2 $}
	      \noLine
	      \UIC{$ M,\Box s, \Box r;r \vdash  s; q,\Box s, N  $}
	      \RL{R$ \Box^; $}
	      \UIC{$ M,r,\Box s, \Box r \vdash  q; s,\Box s, N  $}
	      \RL{R$\rightarrow $}
	      \UIC{$ M,r,\Box s \vdash  q;\Box r\rightarrow s,\Box s, N  $}
	      \RL{L$ \wedge$,L$\neg $}
	      \UIC{$ M,r\wedge\Box s, \neg (\Box r\rightarrow s) \vdash  q;\Box s, N  $}
	      \RL{L$\wedge $}
	      \UIC{$ M,(r\wedge\Box s)\wedge (\neg (\Box r\rightarrow s)) \vdash  q;\Box s, N   $}
	      \RL{L$\Diamond $}
	      \UIC{$ M,\Diamond((r\wedge\Box s)\wedge (\neg (\Box r\rightarrow s))) \vdash  q,\Box s, N  $}
	      \RL{R$\vee $}
	      \UIC{$ M,\Diamond((r\wedge\Box s)\wedge (\neg (\Box r\rightarrow s))) \vdash  q\vee\Box s, N  $}
	  \RL{R$ \wedge $.}
	  \BIC{$ M,\Diamond((r\wedge\Box s)\wedge (\neg (\Box r\rightarrow s))) \vdash p\wedge (q\vee\Box s), N $}
	\end{prooftree}

	 Therefore by permutation of the rules we get the following derivation $ \D' $ for the conclusion
\begin{prooftree}
\AXC{$ \D_1 $}
\noLine
\UIC{$ M,\Box s,\Box r;r\vdash s; p,N $}
\AXC{$ \D_2 $}
\noLine
\UIC{$ M,\Box s,\Box r;r\vdash s; q,\Box s,N $}
\RL{R$\vee $}
\UIC{$ M,\Box s,\Box r;r\vdash s; q\vee\Box s,N $}
\RL{R$\wedge $}
\BIC{$ M,\Box s,\Box r;r\vdash s; p\wedge (q\vee\Box s),N $}
\RL{R$\Box $}
\UIC{$ M,r,\Box s,\Box r\vdash s, \Box(p\wedge (q\vee\Box s)),N $}
\RL{R$\rightarrow $}
\UIC{$ M,r,\Box s\vdash\Box r\rightarrow s, \Box(p\wedge (q\vee\Box s)),N $}
\RL{L$\wedge. $}
\UIC{$ M,r\wedge\Box s\vdash\Box r\rightarrow s, \Box(p\wedge (q\vee\Box s)),N $}	
\end{prooftree}
}
	\end{Examp}
\begin{Cor}\label{G;}
The following rules are admissible.
	\begin{center}
	\AXC{$ M,\Gamma_2;\Gamma_1,\Gamma_3\vdash \Delta_3,\Delta_1;\Delta_2,N $}
	\RL{R$\Box^{G;} $}
	\UIC{$M,\Gamma_1;\Gamma_2,\Gamma_3\vdash \Delta_3,\Delta_2;\Delta_1,N $}
	\DP
	$ \quad $
	\AXC{$M,\Gamma_2;\Gamma_1,\Gamma_3\vdash \Delta_3,\Delta_1;\Delta_2,N  $}
	\RL{L$\Diamond^{G;} $}
	\UIC{$ M,\Gamma_1;\Gamma_2,\Gamma_3\vdash \Delta_3,\Delta_2;\Delta_1,N $}
	\DP
\end{center}
where $ \Diamond\bigwedge(\Gamma_2,\neg \Delta_2) $ is the principle formula  in R$\Box^{G;} $ and $ \Box\bigvee(\neg \Gamma_2,\Delta_2) $ is the principle formula  in L$\Diamond^{G;} $.
\end{Cor}
\section{Admissibility of the Cut rule}\label{sec cut}
In this section, we prove the admissibility of the cut rule and the completeness theorem.

The admissibility of the cut rule,
\begin{prooftree}
	\AXC{$ \Gamma\vdash\Delta,D $}
	\AXC{$ D,\Gamma'\vdash\Delta' $}
	\RL{Cut,}
	\BIC{$ \Gamma,\Gamma'\vdash\Delta,\Delta' $}
	\end{prooftree}
is proved  simultaneously with the following rule
	\begin{prooftree}
	\AXC{$ M,\Gamma_1;\Gamma_2,\Gamma_3\vdash \Delta_3, \Delta_2,D;\Delta_1,N $}
	\RL{$ \text{Cut}^; $}
	\AXC{$ M',\Gamma'_1; D,\Gamma'_2,\Gamma'_3\vdash \Delta'_3, \Delta'_2;\Delta'_1,N' $}
	\BIC{$ M,M',\Gamma_2,\Gamma'_2;\Gamma_1,\Gamma'_1,\Gamma_3,\Gamma'_3\vdash \Delta_3,\Delta'_3, \Delta_1,\Delta_1';\Delta_2,\Delta'_2,N,N' $}
\end{prooftree}
where $ \Diamond\bigwedge(\Gamma_2,\neg \Delta_2, \neg D) $ is in the antecedent or $ \Box\bigvee(\neg \Gamma_2, D,\Delta_2) $ is in the succedent in the left premise, and $ \Diamond\bigwedge(D,\Gamma_2',\neg \Delta'_2) $  is in the antecedent or $ \Box\bigvee(\neg D,\neg  \Gamma_2,\Delta_2) $ in the right premise. Note  the permutation of $ \Gamma_1,\Gamma_2 $, of $ \Gamma'_1,\Gamma'_2 $, of $ \Delta_1,\Delta_2 $ and of $ \Delta'_1,\Delta'_2 $ in the rule.

We can consider the rule $ \text{Cut}^; $ as a new version of the rule cut, where its admissibility  is proved by  admissibility of the cut rule as follows:
	\begin{prooftree}
	\AXC{$ M,\Gamma_1;\Gamma_2,\Gamma_3\vdash \Delta_3, \Delta_2,D;\Delta_1,N $}
	\RL{R$ \Box^{G;}(\ref{G;}) $}
	\UIC{$ M,\Gamma_2;\Gamma_1,\Gamma_3\vdash \Delta_3, \Delta_1;\Delta_2,D,N $}
	   
	   \AXC{$ M',\Gamma'_1; D,\Gamma'_2,\Gamma'_3\vdash \Delta'_3, \Delta'_2;\Delta'_1,N' $}
	   \RL{R$ \Box^{G;}(\ref{G;}) $}
	   \UIC{$ M',D,\Gamma'_2; \Gamma'_1,\Gamma'_3\vdash \Delta'_3, \Delta'_1;\Delta'_1,N' $}
	   \RL{Cut}
	    \BIC{$ M,M',\Gamma_2,\Gamma'_2;\Gamma_1,\Gamma'_1,\Gamma_3,\Gamma'_3\vdash \Delta_3,\Delta'_3, \Delta_1,\Delta_1';\Delta_2,\Delta'_2,N,N' $}
\end{prooftree}
	in which  the formulas $ \Diamond\bigwedge(\Gamma_1,\neg \Delta_1) $ and $ \Diamond\bigwedge( \Gamma'_1,\neg \Delta'_1) $	are principal in  the general version of  the rules R$ \Box $, respectively.
	Similarly, the admissibility of the rule $ \text{Cut}^; $ concludes  the admissibility of the rule cut as follows
	
	\begin{prooftree}
		\AXC{$ \Gamma;\vdash;\Delta,D $}
		\RL{R$ \Box^{G;}(\ref{G;}) $}
		\UIC{$ ;\Gamma\vdash\Delta,D; $}
		    	\AXC{$ \Gamma';\vdash;\Delta',D $}
		    \RL{R$ \Box^{G;}(\ref{G;}) $}
		    \UIC{$ ;\Gamma'\vdash\Delta',D; $} 
	\RL{$ \text{Cut}^; $}
	\BIC{$ \Gamma,\Gamma';\vdash ;\Delta,\Delta' $}
		\end{prooftree}
\begin{The}\label{cut}
	The rules 
	\begin{prooftree}
		\AXC{$ \Gamma;P\vdash Q;\Delta,D $}
		\RL{\rm Cut}
		\AXC{$ D,\Gamma'; P'\vdash Q';\Delta' $}
		\BIC{$ \Gamma,\Gamma';P,P'\vdash Q,Q';\Delta,\Delta' $}
		\end{prooftree}
	and 
	\begin{prooftree}
	\AXC{$ M,\Gamma_1;P_2,P_3\vdash Q_3, Q_2,D;\Delta_1,N $}
	\RL{$ \text{\rm Cut}^; $}
	\AXC{$ M',\Gamma'_1; D,P'_2,P'_3\vdash Q'_3, Q'_2;\Delta'_1,N' $}
	\BIC{$ M,M',P_2,P'_2;\Gamma_1,\Gamma'_1,P_3,P'_3\vdash Q_3,Q'_3, \Delta_1,\Delta_1';Q_2,Q'_2,N,N' $}
\end{prooftree}
	\end{The}
are admissible, where all formulas in the middle parts of the premises, specially the formula $ D $ in the second cut,  are atomic. 
\begin{proof}
	 The both rules are proved  simultaneously  by a main induction
	on   the complexity of the cut formula $ D $ with a  subinduction on the sum of heights of derivations of the two premises (cut-height).  We prove  the first cut rule, the second  is  proved similarly except some cases that we consider at the end of the proof.
	
	If both of the premises are  axioms, then the conclusion is an axiom, and if only one of the premises is an axiom, then the conclusion is obtained by  weakening (see Lemma \ref{W}). 
	
	If one of the last rules in derivations of the premises is  neither 
	L$\Diamond $
	nor
	R$\Box $,
	then the cut will be transformed to simpler cuts as usual. Note that cut-height can increase in the transformation, but the cut formula
	is reduced. For example, let $ D=\Diamond A $ be the principal formula in the left premise.
	\begin{prooftree}
		\AXC{$ \D_1 $}
		\noLine
		\UIC{$ \Gamma;P\vdash Q; \Delta,\Diamond A,A $}
		\RL{R$\Diamond $}
		\UIC{$ \Gamma;P\vdash Q; \Delta,\Diamond A $}
		           \AXC{$ \D_2 $}
		           \noLine
		           \UIC{$\Diamond A,\Gamma'; P'\vdash Q';\Delta' $}
	\RL{Cut,}
	\BIC{$ \Gamma,\Gamma';P,P'\vdash Q,Q';\Delta,\Delta' $}	           
		\end{prooftree}
This cut is transformed into
	\begin{prooftree}
		\AXC{$ \D_1 $}
		\noLine
		\UIC{$ \Gamma;P\vdash Q; \Delta,\Diamond A,A $\hspace{-0.2cm}}
		\AXC{$ \D_2 $}
		\noLine
		\UIC{$ \Diamond A,\Gamma'; P'\vdash Q';\Delta' $}
		\RL{ Cut }
		\BIC{$ \Gamma,\Gamma';P,P'\vdash Q,Q';\Delta,\Delta',A $}
		            \AXC{$ \D_2 $}
		            \noLine
		            \UIC{$ \Diamond A,\Gamma'; P'\vdash Q';\Delta' $}
		            \RL{Lemma \ref{*}}
		            \UIC{$  A,\Gamma'; P'\vdash Q';\Delta' $}
		            \RL{Cut}
		 \BIC{$ \Gamma,\Gamma',\Gamma';P,P',P'\vdash Q,Q',Q';\Delta,\Delta',\Delta' $}
		 \RL{C,}
		 \UIC{$\Gamma,\Gamma';P,P'\vdash Q,Q';\Delta,\Delta' $}
		\end{prooftree}
	where C is used for contraction rules.
	
Let
	$ D=\Box A $
be the principal formula in the right premise.
	\begin{prooftree}
		\AXC{$ \D_1 $}
		\noLine
		\UIC{$ \Gamma;P\vdash Q; \Delta,\Box A $}
		            \AXC{$ \D_2 $}
		            \noLine
		            \UIC{$ A,\Box A,\Gamma'; P'\vdash Q';\Delta' $}
		            \RL{L$\Box $}
		            \UIC{$ \Box A,\Gamma'; P'\vdash Q';\Delta' $}
		\RL{Cut,}
		\BIC{$ \Gamma,\Gamma';P,P'\vdash Q,Q';\Delta,\Delta' $}
		\end{prooftree}
It is transformed into
	\begin{small}
	\begin{prooftree}
	\AXC{$ \D_1 $}
	\noLine
	\UIC{$\Gamma;P\vdash Q; \Delta,\Box A$}	
	\RL{Lemma \ref{*}}
	\UIC{$\Gamma;P\vdash Q; \Delta, A$}
	                 	\AXC{$\D_1$}
	                 \noLine
	                 \UIC{$\Gamma;P\vdash Q; \Delta,\Box A$}
	                 \AXC{$\D_2$}
	                 \noLine
	                 \UIC{$A,\Box A,\Gamma'; P'\vdash Q';\Delta'$}
	                \RL{Cut} 
	                 \BIC{$A,\Gamma,\Gamma';P,P'\vdash Q,Q';\Delta,\Delta'$}  
	 \RL{Cut} 
	 \BIC{$\Gamma,\Gamma,\Gamma';P,P,P'\vdash Q,Q,Q';\Delta,\Delta,\Delta'$} 
	 \RL{C.}  
	 \UIC{$\Gamma,\Gamma';P,P'\vdash Q,Q';\Delta,\Delta'$}                        
		\end{prooftree}
	\end{small}
	$ \, $\\
If 
$ D=A\wedge B $,
$ D=A\vee B $,
$ D=A\rightarrow B $,  or
$ D=\neg A $, then
 by Lemma \ref{inversion}, we use simpler cuts on
 $ A $ and
 $ B $.
 Therefore, in the rest of the proof, we just consider the  cases in which the last rules are 
 L$\Diamond $ or R$\Box $.
 
Case 1.	The cut formula $ D $ is modal.\\
Subcase 1.1. Let
	$ \Delta=Q_1,N,\Box A $ and 
	$ \Delta'=Q'_1,N',\Box B $, and let  the last rules in derivations of the premises be
	R$\Box $ with principal formulas
	$ \Box A $
	and
	$ \Box B $.
	\begin{prooftree}
		\AXC{$ \D_1 $}
		\noLine
		\UIC{$ M;P_1,P\vdash Q,Q_1;N, A,D $}
		\RL{R$\Box $}
		\UIC{$ M,P_1;P\vdash Q;Q_1,N, \Box A,D $}
		           \AXC{$ \D_2 $}
		           \noLine
		           \UIC{$ D,M';P'_1,P'\vdash Q',Q'_1;N', B $}
		           \RL{R$\Box $}
		           \UIC{$  D,M',P'_1;P'\vdash Q';Q'_1,N', \Box B $}
		\RL{Cut,}
		\BIC{$ M,P_1,M',P'_1;P,P'\vdash Q,Q'; Q_1,N,Q'_1,N',\Box A,\Box B $}
		\end{prooftree}
where 	$ \Gamma=M,P_1 $ and $ \Gamma'=M',P'_1 $. This cut is transformed into
	\begin{prooftree}
		\AXC{$ \D_1 $}
		\noLine
		\UIC{$ M;P_1,P\vdash Q,Q_1;N, A,D $}
		              \AXC{$ \D_2 $}
		              \noLine
		              \UIC{$D,M';P'_1,P'\vdash Q',Q'_1;N', B $}
		              \RL{R$\Box $}
		              \UIC{$ D,M';P'_1,P'\vdash Q',Q'_1;N', \Box B $}
		 \RL{Cut}
		 \BIC{$ M,M';P_1,P,P'_1,P'\vdash Q,Q_1, Q',Q'_1;N,N', A,\Box B $}
		 \RL{R$\Box $.}
		 \UIC{$ M,P_1,M',P'_1;P,P'\vdash Q,Q';Q_1,N,\Box A,Q'_1,N',\Box B $}
		\end{prooftree}
Subcase 1.2. Let
	$ \Gamma=\Diamond A,M,P_1$ and  $\Delta'=Q'_1,N',\Box B $, and let the last rules  be L$\Diamond $ and R$\Box $ with principal formulas
	$\Diamond A $
	and
	$ \Box B $:
	\begin{prooftree}
		\AXC{$ \D_1 $}
		\noLine
		\UIC{$ A,M;P_1,P\vdash Q,Q_1;N_1, D $}
		\RL{L$\Diamond $}
		\UIC{$  \Diamond A,M,P_1;P\vdash Q;Q_1,N, D  $}
		              \AXC{$ \D_2 $}
		              \noLine
		              \UIC{$D,M';P'_1,P'\vdash Q',Q'_1;N', B$}
		              \RL{R$\Box$}
		              \UIC{$D,M',P'_1;P'\vdash Q';Q'_1,N', \Box B$}  
		  \RL{Cut,}
		  \BIC{$ \Diamond A,M,P_1,M',P'_1;P,P'\vdash Q,Q';Q_1,N,Q'_1,N',\Box B $}
		\end{prooftree}
where 	$ \Delta=Q_1,N $ and $ \Gamma'=M',P'_1$. This cut is transformed into
	\begin{prooftree}
		\AXC{$ \D_1 $}
		\noLine
		\UIC{$  A,M;P_1,P\vdash Q,Q_1;N, D $}
		       \AXC{$ \D_2 $}
		       \noLine
		       \UIC{$ D,M';P'_1,P'\vdash Q',Q'_1;N', B $}
		       \RL{R$\Box $}
		       \UIC{$  D,M';P'_1,P'\vdash Q',Q'_1;N', \Box B $}
	  \RL{Cut}
	  \BIC{$ A,M,M';P_1,P'_1,P,P'\vdash Q,Q_1, Q_1,Q'_1;N,N', \Box B $}
	  \RL{L$\Diamond $.}
	  \UIC{$ \Diamond A,M,P_1,M',P'_1;P,P'\vdash Q,Q';Q_1,N,Q'_1,N',\Box B $}
		\end{prooftree}
Subcase 1.3. Let 
	$ \Gamma=\Diamond A,M,P_1$ and  $\Gamma'=\Diamond B,M',P'_1 $, and let the last rules  be L$\Diamond $  with principal formulas
	$\Diamond A $
	and
	$ \Diamond B $:
	\begin{prooftree}
		\AXC{$ \D_1 $}
		\noLine
		\UIC{$  A,M;P_1,P\vdash Q,Q_1;N, D $}
		\RL{L$\Diamond $}
		\UIC{$ \Diamond A,M,P_1;P\vdash Q;Q_1,N, D $}
		              \AXC{$ \D_2 $}
		              \noLine
		              \UIC{$  D, B,M';P'_1,P'\vdash Q',Q'_1;N'  $}
		              \RL{L$\Diamond $}
		              \UIC{$ D,\Diamond B,M',P'_1;P'\vdash Q';Q'_1,N' $}
		\RL{Cut,}
		\BIC{$ \Diamond A,M,P_1, \Diamond B,M',P'_1;P,P'\vdash Q,Q';Q_1,N,Q'_1,N' $}
		\end{prooftree}
where  	$ \Delta=N,Q_1 $ and $ \Delta'=N',Q'_1 $. This cut is transformed into
	\begin{prooftree}
		\AXC{$ \D_1 $}
		\noLine
		\UIC{$  A,M;P_1,P\vdash Q,Q_1;N, D $}
		               \AXC{$ \D_2 $}
		               \noLine
		               \UIC{$ D, B,M';P'_1,P'\vdash Q',Q'_1;N'  $}
		               \RL{L$\Diamond $}
		               \UIC{$ D,\Diamond B,M';P'_1,P'\vdash Q',Q'_1;N' $}
		 \RL{ Cut }
		 \BIC{$  A,M, \Diamond B,M';P_1,P,P'_1,P'\vdash Q,Q_1,Q',Q'_1;N,N' $}
		 \RL{L$\Diamond $.}
		 \UIC{$ \Diamond A,M,P_1, \Diamond B,M',P'_1;P,P'\vdash Q,Q';Q_1,N,Q'_1,N' $}
		\end{prooftree}
Subcase 1.4. Let
	$ \Gamma'=\Diamond B,M',P'_1$ and  $\:\Delta=Q_1,N,\Box A $, and let  the last rules  be R$\Box $ and L$\Diamond $ with principal formulas
	$\Box A $
	and
	$ \Diamond B $:
	\begin{prooftree}
		\AXC{$ \D_1 $}
		\noLine
		\UIC{$ M;P_1,P\vdash Q,Q_1;N,A, D $}
		\RL{R$\Box $}
		\UIC{$ M,P_1;P\vdash Q;Q_1,N,\Box A, D $}
		             \AXC{$\D_2  $}
		             \noLine
		             \UIC{$ D, B,M';P'_1,P'\vdash Q',Q'_1;N'  $}
		             \RL{L$\Diamond $}
		             \UIC{$ D,\Diamond B,M',P'_1;P'\vdash Q';Q'_1,N' $}
		\RL{Cut,}
		\BIC{$  M,P_1, \Diamond B,M',P'_1;P,P'\vdash Q,Q';Q_1,N,\Box A,Q'_1,N' $}
		\end{prooftree}
where 	$ \Gamma=M,P_1$ and  $\Delta'=N',Q'_1 $. This cut is transformed into
\begin{prooftree}
	\AXC{$ \D_1 $}
	\noLine
	\UIC{$  M;P_1,P\vdash Q,Q_1;N,A, D $}
	                \AXC{$ \D_2 $}
	                \noLine
	                \UIC{$ D, B,M';P'_1,P'\vdash Q',Q'_1;N'  $}
	                \RL{L$\Diamond $}
	                \UIC{$ D,\Diamond B,M';P'_1,P'\vdash Q',Q'_1;N' $}
	 \RL{Cut}
	 \BIC{$  M, \Diamond B,M';P_1,P,P'_1,P'\vdash Q,Q_1,Q',Q'_1;N,A,N' $}
	 \RL{R$\Box $.}
	 \UIC{$ M,P_1, \Diamond B,M',P'_1;P,P'\vdash Q,Q';Q_1,N,\Box A,Q'_1,N'$}
	\end{prooftree}
 Case 2. The cut formula  $ D $ is atomic.
In this case we only consider the case that when last rules in derivations of the premises are R$ \Box $ and L$ \Diamond $; the other cases are proved by similar argument.\\
 Let
$ \Gamma'=\Diamond B,M',P'_1$ and  $\Delta=Q_1,N,\Box A $,  and let the last rules in derivations of the premises be L$\Diamond $  and R$\Box $ with principal formulas
$ \Diamond B $
and
$\Box A $:
	\begin{prooftree}
		\AXC{$ \D_1 $}
		\noLine
		\UIC{$ M;P_1,P\vdash Q,D,Q_1;N,A $}
		\RL{R$\Box $}
		\UIC{$ M,P_1;P\vdash Q;Q_1,N,\Box A, D $}
		    \AXC{$ \D_2 $}
		    \noLine
	    	\UIC{$  B,M';P'_1,D,P'\vdash Q',Q'_1;N'  $}
		    \RL{L$\Diamond $}
		    \UIC{$ D,\Diamond B,M',P'_1;P'\vdash Q';Q'_1,N' $}
		\RL{Cut.}
		\BIC{$  M,P_1, \Diamond B,M',P'_1;P,P'\vdash Q,Q';Q_1,N,\Box A,Q'_1,N' $}
	\end{prooftree}
This is transformed into
\begin{prooftree}
	\AXC{$ \D_1 $}
	\noLine
	\UIC{$ M;P_1,P\vdash Q,D,Q_1;N,A $}
	\AXC{$ \D_2 $}
	\noLine
	\UIC{$  B,M';P'_1,D,P'\vdash Q',Q'_1;N'  $}
	\RL{$ \text{Cut}^; $,}
	\BIC{$  M,P_1,M',P'_1;B,P,P'\vdash Q,Q',A; Q_1,N,Q'_1,N' $}
	\UIC{$  M,P_1, \Diamond B,M',P'_1;P,P'\vdash Q,Q';Q_1,N,\Box A,Q'_1,N' $}
\end{prooftree}
Finally, for the second rule, we consider some cases.\\
Case 1. If one of  the  last rule in derivation of  the premises are not modal rules, then the cut is transformed into simpler cut(s) and then by applying Lemma \ref{sadeh}, the conclusion is obtained. In the following we consider some cases.\\
Subcase 1.1. Let  $ \Gamma'_1=\Gamma'_{1_1},A\wedge B $, and let $ A\wedge B $   be the principal formula in the right premise.
	\begin{prooftree}
	\AXC{$ \D_1 $}
	\noLine
	\UIC{$ M,\Gamma_1;P_2,P_3\vdash Q_3, Q_2,D;\Delta_1,N $}	
	       \AXC{$ \D_2 $}
	       \noLine
	       \UIC{$ M',\Gamma'_{1_1},A, B; D,P'_2,P'_3\vdash Q'_3, Q'_2;\Delta'_1,N' $}
	     \RL{$ L\wedge $}
	    \UIC{$ M',\Gamma'_{1_1},A\wedge B; D,P'_2,P'_3\vdash Q'_3, Q'_2;\Delta'_1,N' $}
	    \RL{$ \text{Cut}^; $}
	    \BIC{$ M,M',P_2,P'_2;\Gamma_1,\Gamma'_{1_1},A\wedge B,P_3,P'_3\vdash Q_3,Q'_3, \Delta_1,\Delta_1';Q_2,Q'_2,N,N' $}
\end{prooftree}
This cut is transformed into
\begin{prooftree}
	\AXC{$ \D_1 $}
	\noLine
	\UIC{$ M,\Gamma_1;P_2,P_3\vdash Q_3, Q_2,D;\Delta_1,N $}
	    \AXC{$ \D_2 $}
	    \noLine
	     \UIC{$ M',\Gamma'_{1_1},A, B; D,P'_2,P'_3\vdash Q'_3, Q'_2;\Delta'_1,N' $} 
	  \RL{$ \text{Cut}^; $}
	   \BIC{$ M,M',P_2,P'_2;\Gamma_1\Gamma'_{1_1},A, B,P_3,P'_3\vdash Q_3,Q'_3, \Delta_1,\Delta_1';Q_2,Q'_2,N,N' $}
	   \RL{Lemma \ref{sadeh.inv}}
	   \UIC{$ M,M',P_2,P'_2;\Gamma_1,\Gamma'_{1_1},A\wedge B,P_3,P'_3\vdash Q_3,Q'_3, \Delta_1,\Delta_1';Q_2,Q'_2,N,N' $}
	\end{prooftree}
Subcase 1.2. Let  $ \Gamma'_1=\Gamma'_{1_1},A\vee B $, and let $ A\vee B $ be the principal formula.
	\begin{prooftree}
		\AXC{$ \D_1 $}
		\noLine
		\UIC{$ M,\Gamma_1;P_2,P_3\vdash Q_3, Q_2,D;\Delta_1,N $}
	    	\AXC{$ \D_2 $}
		    \noLine
			\UIC{$ M',\Gamma'_{1_1},A\vee B; D,P'_2,P'_3\vdash Q'_3, Q'_2;\Delta'_1,N' $}  
	\RL{$ \text{Cut}^; $,}
	\BIC{$ M,M',P_2,P'_2;\Gamma_{1_1},\Gamma'_{1_1},A\vee B,P_3,P'_3\vdash Q_3,Q'_3, \Delta_1,\Delta_1';Q_2,Q'_2,N,N' $}
\end{prooftree}
where $ \D_2 $ is as follows
\begin{prooftree}
	\AXC{$ \D_{2_1} $}
	\noLine
	\UIC{$  M',\Gamma'_{1_1},A; D,P'_2,P'_3\vdash Q'_3, Q'_2;\Delta'_1,N'  $}	
	   \AXC{$ \D_{2_2} $}
	    \noLine
	    \UIC{$ M',\Gamma'_{1_1}, B; D,P'_2,P'_3\vdash Q'_3, Q'_2;\Delta'_1,N' $}
	\RL{L$ \vee $}
	\BIC{$ M',\Gamma'_{1_1},A\vee B; D,P'_2,P'_3\vdash Q'_3, Q'_2;\Delta'_1,N' $}
	\end{prooftree}
This cut is transformed  into two cuts:
\begin{prooftree}
	\AXC{$ \D_1 $}
	\noLine
	\UIC{$ M,\Gamma_1;P_2,P_3\vdash Q_3, Q_2,D;\Delta_1,N $}
      	\AXC{$ \D_{2_1} $}
	     \noLine
	      \UIC{$  M',\Gamma'_{1_1},A; D,P'_2,P'_3\vdash Q'_3, Q'_2;\Delta'_1,N'  $}	
	\RL{$ \text{Cut}^; $}
	\BIC{$ M,M',P_2,P'_2;\Gamma_1,\Gamma'_{1_1},A,P_3,P'_3\vdash Q_3, Q'_3,\Delta_1,\Delta'_1;Q_2, Q'_2;N,N' $}
\end{prooftree}
and
\begin{prooftree}
\AXC{$ \D_1 $}
\noLine
\UIC{$ M,\Gamma_1;P_2,P_3\vdash Q_3, Q_2,D;\Delta_1,N $}
     \AXC{$ \D_{2_2} $}
     \noLine
      \UIC{$ M',\Gamma'_{1_1}, B; D,P'_2,P'_3\vdash Q'_3, Q'_2;\Delta'_1,N' $}	
\RL{$ \text{Cut}^; $}
\BIC{$ M,M',P_2,P'_2;\Gamma_1,\Gamma'_{1_1},B,P_3,P'_3\vdash Q_3, Q'_3,\Delta_1,\Delta'_1;Q_2, Q'_2;N,N' $}
\end{prooftree}
Therefore by applying Lemma \ref{sadeh.inv}  the conclusion is obtained.\\
Subcase 1.3. Let  $ \Gamma'_1=\Gamma'_{1_1},A\rightarrow B $, and let $ A\rightarrow B $ be the principal formula.
\begin{prooftree}
	\AXC{$ \D_1 $}
	\noLine
	\UIC{$ M,\Gamma_1;P_2,P_3\vdash Q_3, Q_2,D;\Delta_1,N $}
	    \AXC{$ \D_2 $}
	    \noLine
	    \UIC{$ M',\Gamma'_{1_1},A\rightarrow B; D,P'_2,P'_3\vdash Q'_3, Q'_2;\Delta'_1,N' $}  
	\RL{$ \text{Cut}^; $,}
	\BIC{$ M,M',P_2,P'_2;\Gamma_{1_1},\Gamma'_{1_1},A\rightarrow B,P_3,P'_3\vdash Q_3,Q'_3, \Delta_1,\Delta_1';Q_2,Q'_2,N,N' $}
\end{prooftree}
where $ \D_2 $ is as follows
\begin{prooftree}
	\AXC{$ \D_{2_1} $}
	\noLine
	\UIC{$  M',\Gamma'_{1_1}; D,P'_2,P'_3\vdash Q'_3, Q'_2;\Delta'_1,N',A  $}	
	    \AXC{$ \D_{2_2} $}
    	\noLine
    	\UIC{$ B,M',\Gamma'_{1_1}; D,P'_2,P'_3\vdash Q'_3, Q'_2;\Delta'_1,N' $}
	\RL{L$ \rightarrow $}
	\BIC{$ M',\Gamma'_{1_1},A\rightarrow B; D,P'_2,P'_3\vdash Q'_3, Q'_2;\Delta'_1,N' $}
\end{prooftree}
This cut is transformed  into two cuts:
\begin{prooftree}
	\AXC{$ \D_1 $}
	\noLine
	\UIC{$ M,\Gamma_1;P_2,P_3\vdash Q_3, Q_2,D;\Delta_1,N $}
    	\AXC{$ \D_{2_1} $}
    	\noLine
    	\UIC{$  M',\Gamma'_{1_1}; D,P'_2,P'_3\vdash Q'_3, Q'_2;\Delta'_1,N',A  $}	
	\RL{$ \text{Cut}^; $}
	\BIC{$ M,M',P_2,P'_2;\Gamma_1,\Gamma'_{1_1},P_3,P'_3\vdash Q_3, Q'_3,\Delta_1,\Delta'_1;Q_2, Q'_2;N,N',A $}
\end{prooftree}
and
\begin{prooftree}
	\AXC{$ \D_1 $}
	\noLine
	\UIC{$ M,\Gamma_1;P_2,P_3\vdash Q_3, Q_2,D;\Delta_1,N $}
	   \AXC{$ \D_{2_2} $}
	   \noLine
	   \UIC{$ B,M',\Gamma'_{1_1}; D,P'_2,P'_3\vdash Q'_3, Q'_2;\Delta'_1,N' $}	
	\RL{$ \text{Cut}^; $}
	\BIC{$ M,B,M',P_2,P'_2;\Gamma_1,\Gamma'_{1_1},P_3,P'_3\vdash Q_3, Q'_3,\Delta_1,\Delta'_1;Q_2, Q'_2;N,N' $}
\end{prooftree}
Therefore by applying Lemma \ref{sadeh.inv}  the conclusion is obtained.

Case 2. If  the  last rules in derivation of the premises are modal rules and in one of them,
$ D $ does not occur in   the principal formula, then the cut is transformed into a simpler cut. As a typical example, let the last rule in derivation of the premise is
 L$\Diamond $ be as follows
	\begin{prooftree}
	\AXC{$ M,\Gamma_1;P_2,P_3\vdash Q_3, Q_2,D;\Delta_1,N $}
     	\AXC{$ \D_2 $}
    	\noLine
	    \UIC{$B, M'_1;P'_1, D,P'_2,P'_3\vdash Q'_3, Q'_2,Q'_1;N' $}
	    \RL{L$ \Diamond $}
    	\UIC{$ \Diamond B,M'_1,P'_1; D,P'_2,P'_3\vdash Q'_3, Q'_2;Q'_1,N' $}
	\RL{$ Cut^; $,}
	\BIC{$ M,\Diamond B,M'_1,P_2,P'_2;\Gamma_1,P'_1,P_3,P'_3\vdash Q_3,Q'_3, \Delta_1,Q_1';Q_2,Q'_2,N,N' $}
\end{prooftree}

where  $ M'=\Diamond B,M'_1 $ and  $ \Gamma'_1=P'_1 $,$ \,\Delta'_1=Q'_1 $ are atomic. This cut is transformed into
\begin{prooftree}
	\AXC{$ M,P_1;P_2,P_3\vdash Q_3, Q_2,D;Q_1,N $}
	       \AXC{$ \D_2 $}
	       \noLine
	       \UIC{$B, M';P'_1, D,P'_2,P'_3\vdash Q'_3, Q'_2,Q'_1;N' $}
	       \RL{$ \text{Cut}^;. $}
	 \BIC{$ M,M',P_2,P'_2;B,P_1,P'_1,P_3,P'_3\vdash Q_3,Q'_3, Q_1,Q_1';Q_2,Q'_2,N,N' $}
	 \RL{$ \ast $}
	 \UIC{$ M,\Diamond B,M',P_2,P'_2;P_1,P'_1,P_3,P'_3\vdash Q_3,Q'_3, Q_1,Q_1';Q_2,Q'_2,N,N' $}
	\end{prooftree} 
where the conclusion in the  rule $ \ast $ is a rewriting of its premise. 

Case 3. Let the last rules in derivations of the premises be $  L\Diamond^;$ or R$\Box^; $ in which $ D $ occurs in both principal formulas. For example let the following derivation with the principal formulas 
$ \Diamond\bigwedge( P_2,\neg Q_2,\neg D) $ and $\Box\bigvee(\neg D,\neg P'_2, Q'_2) $, be respectively:
\begin{prooftree}
	\AXC{$\D_1 $}
	\noLine
	\UIC{$M,P_2;P_1,P_3\vdash Q_3,Q_1;Q_2,D,N$}
	\RL{L$\Diamond^;$}
	\UIC{$M,P_1;P_2,P_3\vdash Q_3,Q_2,D;Q_1,N$}
	\AXC{$\D_2$}
	\noLine
	\UIC{$M',D,P'_2;P'_1,P'_3\vdash Q'_3,Q'_1;Q'_2,N'$}
	\RL{R$\Box^;$}
	\UIC{$M',P'_1;D,P'_2,P'_3\vdash Q'_3,Q'_2;Q'_1,N'$}
	\RL{ $ \text{Cut}^; $.}
	\BIC{$ M,M',P_2,P'_2;P_1,P'_1,P_3,P'_3\vdash Q_3,Q'_3,Q'_1,Q'_1;Q_2,Q'_2,N,N'$}
\end{prooftree}

This cut rule is transformed into the first cut as follows:
\begin{prooftree}
	\AXC{$\D_1 $}
	\noLine
	\UIC{$M,P_2;P_1,P_3\vdash Q_3,Q_1;Q_2,D,N$}
	    	\AXC{$\D_2$}
	        \noLine
	        \UIC{$M',D,P'_2;P'_1,P'_3\vdash Q'_3,Q'_1;Q'_2,N'$}
    \RL{Cut}
    \BIC{$ M,M',P_2,P'_2;P_1,P'_1,P_3,P'_3\vdash Q_3,Q'_3,Q'_1,Q'_1;Q_2,Q'_2,N,N'$}
\end{prooftree}
The other cases are proved similarly.
\end{proof}
\begin{The}\label{complete}
	The following are equivalent.
	\begin{itemize}
		\item[{\rm (1)}]
		The sequent
		$ \Gamma\vdash A $ is S5-valid.
		\item[{\rm (2)}]
	$ \Gamma\vdash_{\text{S5}} A $.
		\item[{\rm (3)}]
		The sequent
	$ \Gamma\vdash A $
	is provable in $ \text{\rm G3{\scriptsize S5}}^;  $.
\end{itemize}
\end{The}
\begin{proof}
	 (1) implies  (2)  by completeness of S5.
	 (3) implies (1)  by soundness of $ \text{\rm G3{\scriptsize S5}}^;  $. We show that (2) implies (3).
	Suppose 
	$ A_1,\ldots,A_n $
	is  an S5-proof  of $ A $ from $ \Gamma $. This means that $ A_n $ is $ A $ and that each $ A_i $ is in $ \Gamma $, is an axiom, or is inferred by  modus ponens or necessitation. It is straightforward to
	prove, by induction on $ i $, that $ \Gamma\vdash A_i $ for each $ A_i $.
	
	Case 1. $ A_i\in \Gamma $: It is a direct consequence of Lemma \ref{arbit}.
	
	Case 2. $ A_i $ is an axiom of S5: All axioms of S5 are easily proved in $ \text{\rm G3{\scriptsize S5}}^;  $. As a typical example, in the following we prove the axiom 5:
	\begin{prooftree}
		\AXC{$ \Diamond A\vdash \Diamond A $}
		\RL{R$\Box $}
		\UIC{$ \Diamond A\vdash \Box\Diamond A $}
		\RL{R$\rightarrow $.}
		\UIC{$ \vdash \Diamond A\rightarrow \Box\Diamond A $}
		\end{prooftree}
	
	Case 3. $ A_i $ is inferred by  modus ponens: Suppose  $ A_i $ is inferred from $ A_j $ and $  A_j\rightarrow A_i $, $ j<i $, by use of the cut rule we prove $ \Gamma\vdash A_i $:
	\begin{prooftree}
		\AXC{}
		\RL{IH}
		\UIC{$ \Gamma\vdash A_j $}
		               \AXC{}
		               \RL{ IH }
		               \UIC{$ \Gamma\vdash A_j\rightarrow A_i $}
		                            \AXC{$ A_j\vdash A_i, A_j $}
		                             \AXC{$  A_i\vdash A_i $}
		                             \RL{L$\rightarrow $}
		                             \BIC{$ A_j\rightarrow A_i, A_j\vdash A_i $}
		               \RL{Cut}
		               \BIC{$A_j,\Gamma\vdash A_i  $}
        \RL{Cut}
        \BIC{$ \Gamma,\Gamma\vdash A_i $}
        \RL{LC. }
        \UIC{$ \Gamma\vdash A_i $}
		\end{prooftree}
	Case 4. $ A_i $ is inferred by  necessitation:
Suppose $ A_i=\Box A_j $ is inferred from $ A_j $ by  necessitation. In this case, $ \vdash_{S5}A_j $ (since the rule necessitation  can be applied only to premises which are derivable in the axiomatic system) and so we have:
\begin{prooftree}
\AXC{}
\RL{IH}
\UIC{$ \vdash A_j $}
\RL{R$ \Box $}
\UIC{$ \vdash\Box A_j $}
\RL{W.}
\UIC{$ \Gamma\vdash A_i $}	
\end{prooftree}
	\end{proof}
\begin{Cor}
$ \text{\rm G3{\scriptsize S5}}^;  $ is sound and complete with respect to the S5 {\rm Kripke} frames.
\end{Cor}
\section{Conclusion}
We have presented system G3{\scriptsize S5}, a sequent calculus for S5, this system does not have the subformula property although in a  bottom-up    proof,   the formulas in the premises are constructed by atomic formulas in the conclusions. For convenience we  have rewritten the rules of G3{\scriptsize S5} by using semicolon in system $ \text{\rm G3{\scriptsize S5}}^;  $ which enjoys the subformula property. Also we have proved the completeness theorem and  admissibility of the  weakening, contraction and cut rules in it. All properties which are proved in $ \text{\rm G3{\scriptsize S5}}^;  $ are also proved in G3{\scriptsize S5} because the system  $ \text{\rm G3{\scriptsize S5}}^;  $ is  a rewritten of the system  G3{\scriptsize S5}.

We can consider the system $ \text{\rm G3{\scriptsize S5}}^;  $  primary, since  the  relatedness of formulas in the middle parts of the premises can be determined from theirs in the conclusion of the rules.

\textbf{Acknowledgement}. The authors would like to thank Meghdad Ghari for helpful
discussions.

\newlength{\bibitemsep}\setlength{\bibitemsep}{.2\baselineskip plus .05\baselineskip minus .05\baselineskip}
\newlength{\bibparskip}\setlength{\bibparskip}{0pt}
\let\oldthebibliography\thebibliography
\renewcommand\thebibliography[1]{%
	\oldthebibliography{#1}%
	\setlength{\parskip}{\bibitemsep}%
	\setlength{\itemsep}{\bibparskip}%
}

 \end{document}